\documentclass[11pt, reqno]{amsart}
\usepackage[utf8]{inputenc}
\usepackage{amsmath}
\usepackage{amsfonts}
\usepackage{amssymb}
\usepackage{amsthm}
\usepackage{dsfont}
\usepackage{hyperref}
\usepackage[linesnumbered,lined,boxed,commentsnumbered]{algorithm2e}
\usepackage{tikz-cd}

\DeclareMathOperator{\Hom}{Hom}
\DeclareMathOperator{\Ann}{Ann}
\DeclareMathOperator{\End}{End}
\DeclareMathOperator{\Spec}{Spec}
\DeclareMathOperator{\HH}{H}

\DeclareMathOperator{\id}{id}
\DeclareMathOperator{\im}{im}
\DeclareMathOperator{\pr}{pr}
\DeclareMathOperator{\Pic}{Pic}
\DeclareMathOperator{\Ext}{Ext}
\DeclareMathOperator{\rk}{rk}
\theoremstyle{definition}
\newtheorem{definition}{Definition}[section]
\newtheorem{propdef}[definition]{Definition/Proposition}
\newtheorem{exmp}[definition]{Example}
\newtheorem{rem}[definition]{Remark}
\theoremstyle{theorem}
\newtheorem{thm}[definition]{Theorem}
\newtheorem{cor}[definition]{Corollary}
\newtheorem{lem}[definition]{Lemma}
\newtheorem{prop}[definition]{Proposition}
\begin{document}
\title{Constructing all Genus 2 Curves with Supersingular Jacobian}
\date{\today}
\begin{abstract}
L. Moret-Bailly constructed families $\mathfrak{C}\rightarrow \mathbb{P}^1$ of genus 2 curves with supersingular Jacobian. In this paper we first classify the reducible fibers of a Moret-Bailly family using linear algebra over a quaternion algebra. The main result is an algorithm that exploits properties of two reducible fibers to compute a hyperelliptic model for any irreducible fiber of a Moret-Bailly family.
\end{abstract}

\author[Pieper]{Andreas Pieper}
\address{
  Andreas Pieper,
  Institut für Algebra und Zahlentheorie,
  Universität Ulm,
  Helmholtzstrasse 18,
  89081 Ulm,
  Germany
}
\email{andreas.pieper@uni-ulm.de}

\maketitle
\tableofcontents

\section*{Introduction}
Let $k$ be an algebraically closed field of characteristic $p>0$. Recall that an Abelian variety $A/k$ is called supersingular if it is isogenous to a product of supersingular elliptic curves. The supersingular locus in the moduli space of $g$-dimensional principally polarized Abelian varieties has been intensively studied. K.-Z. Li and F. Oort \cite{Li-Oort} prove that the supersingular locus is birational to a variety explicitly described in terms of semi-linear algebra (Dieudonn\'e theory). This allows one to determine the number of irreducible components of the supersingular locus and shows that it is equidimensional of dimension $\lfloor \frac{g^2}{4} \rfloor$.
\par If $g=2$ this result says more precisely that the supersingular locus is the union of rational curves, as was proven earlier by T. Katsura and Oort in \cite{KatsuraOort}. They prove that every supersingular principally polarized Abelian variety is a fiber of one of the families $\mathcal{A}\rightarrow \mathbb{P}^1$ constructed by Moret-Bailly in \cite{Moret-Bailly}. Therefore the supersingular locus is the union of the images of the $\mathbb{P}^1$s from Moret-Bailly's construction.
\par This construction works as follows for $p>2$ (the case $p=2$ is similar but slightly different): Indeed fix a supersingular elliptic curve $E/k$. Then the kernel of geometric Frobenius is $E[F]\cong \alpha_p$. Since $\Hom(\alpha_p, \alpha_p^2)=k^2$, the subgroup schemes isomorphic to $\alpha_p$ in $E^2$ are parametrized by $\mathbb{P}^1$. Moret-Bailly studies the Abelian scheme $\mathcal{A} \rightarrow \mathbb{P}^1$ obtained as a quotient of $E^2_{\mathbb{P}^1}$ by a certain finite flat group scheme. (Here and in the following for two $k$-schemes $S$ and $T$ the notation $S_T$ is an abbreviation for $S\times_k T$). This puts all the possible ways of taking a quotient of $E^2$ by $\alpha_p$ into a family of Abelian surfaces over $\mathbb{P}^1$.
\par Furthermore Moret-Bailly explains how to equip $\mathcal{A}$ with a principal polarization. For that purpose we need a polarization $\eta$ on $E^2$ with $\ker(\eta)=E^2[F]$. Then Moret-Bailly proves that $\eta_{\mathbb{P}^1}$ descends to a principal polarization $\lambda$ on $\mathcal{A}$. He shows that there exists a relative divisor $\mathfrak{C}\subset \mathcal{A}$ inducing $\lambda$, a ``theta divisor''. The fibers of $\mathfrak{C} \rightarrow \mathbb{P}^1$ are either smooth curves of genus 2 or two elliptic curves intersecting transversely. In this way we get finitely many families of curves of genus 2 corresponding to the equivalence classes of $\eta$. (Here the equivalence relation is conjugation by $\text{Aut}(E^2)$. See \cite[Theorem 5.7]{KatsuraOort} for a discussion of the cardinality of the equivalence classes.) The Jacobians of those curves are supersingular as they are isogenous to $E^2$.
\par The goal of the present paper is to make this theory effective by giving an algorithm that computes all the non-singular fibers of a Moret-Bailly family $\mathfrak{C}\rightarrow \mathbb{P}^1$. For that purpose we need to be able to compute quotients by $\alpha_p$. In the paper \cite{ThetaIsog} D. Lubicz and D. Robert spread the idea that separable isogenies between higher dimensional Abelian varieties should be computed using Mumford's theory of theta groups. This point of view can be extended to inseparable isogenies.
\par Mumford defines for an ample line bundle $\mathcal{L}$ on an Abelian variety $A$, a group scheme $\mathcal{G}(\mathcal{L})$ with center $\mathbb{G}_m$ (the so-called theta group). Furthermore $\mathcal{G}(\mathcal{L})$ acts irreducibly on $\HH^0(A, \mathcal{L})$ such that the center acts by multiplication. But he also shows that there exists only one representation of $\mathcal{G}(\mathcal{L})$ with this property.
\par In the separable case, that is $\text{char}(k)\nmid \deg \mathcal{L}$, Mumford proves that $\mathcal{G}(\mathcal{L})$ is isomorphic to an explicit group $\mathcal{G}(\delta)$, a so-called Heisenberg group. He also writes down an explicit irreducible representation $V(\delta)$ of $\mathcal{G}(\delta)$ such that the center acts by multiplication. From the uniqueness of the representation one gets an isomorphism $\HH^0(A,\mathcal{L})\cong V(\delta)$. That is a very important result and can be used to prove theorems in a purely algebraic fashion that were classically proven over $\mathbb{C}$ via the analytic theory of theta functions.
\par Unfortunately the structure theorem for theta groups fails in the inseparable case. However in our case we are lucky as the theta group we are using is nevertheless a Heisenberg group $\mathbb{G}_m \times \alpha_p \times \alpha_p$ with group law as explained in Definition \ref{DefHeis}.
\par To construct such an isomorphism with a Heisenberg group it turns out that we need two fibers of $\mathfrak{C}\rightarrow \mathbb{P}^1$, one for each copy of $\alpha_p$. Then we exploit that there are plenty ($5p-5>2$) of fibers that are as simple as possible, namely the reducible fibers. For that purpose we develop a classification of the reducible fibers in terms of linear algebra over the quaternion algebra $B=\End(E)\otimes \mathbb{Q}$.
\par Using two reducible fibers also has the advantage that we can construct a rational function $g$ on $E^2$ from them. This rational function will be interpreted as a section of a line bundle on $E^2$, via the correspondence between divisors and line bundles. Then we can construct further rational functions on $E^2$ by translating $g$ using the group structure on $E^2$. In the end the desired curve can be obtained from the zero divisor of such a translate of $g$. We refer to the main body of the text for the details.
\par The outline of the paper is as follows: In chapter 1 we prove some results about Hermitian lattices over maximal orders in quaternion algebras. These are analogous to similar statements for Hermitian lattices over the ring of integers of a number field and the proofs work the same way. In chapter 2 we recall some facts about group schemes related to $\alpha_p$. In chapter 3 we introduce theta groups, which are the cornerstone of our method. The theory is usually formulated in terms of line bundles. Unfortunately we need to prove some technical compatibilities when we translate from line bundles to divisors.
\par Chapter 4 discusses the main theorems. First we classify the reducible fibers of a Moret-Bailly family and give an algorithm that computes them. Then we bootstrap ourselves and use two reducible fibers to compute all the irreducible fibers.\\
\textbf{Acknowledgements}: The material presented here is part of the author's PhD thesis written under the supervision of S. Wewers. I want to thank him for reading the manuscript carefully and suggesting various improvements. The author wishes to thank the anonymous referees for valuable comments. Thanks go to I. Bouw, B. Dina, J. Hanselman, and J. Sijsling for valuable discussions.

%
%
%
%
%

\section{Hermitian Forms over Quaternion Algebras}
In this section let $R$ be a Dedekind ring with field of fractions $K$ and $B$ a quaternion algebra over $K$, i.e., a central simple $K$-algebra with $[B:K]=4$. Denote by $\overline{(\cdot)}$ the standard involution on $B$. Let further $\mathsf{O}$ be a maximal $R$-order of $B$. Later we will only consider the case, where $R=\mathbb{Z}$ or a localization $\mathbb{Z}_{(l)}$ and $B$ is the unique quaternion algebra over $\mathbb{Q}$ ramified only at $(p,\infty)$. However as it does not lead to additional difficulties we will present all the statements in most general form. The goal of this section is to prove the Lemma \ref{PerfPair} about Hermitian forms over $\mathsf{O}$. It says that for a perfect hermitian pairing $(\cdot, \cdot )$ on an $\mathsf{O}$-module $M$ and a primitive vector $v\in M$ there exists a $w\in M$ with $(v,w)=1$. This lemma is analogous to the similar statement over the ring of integers of a number field and the proof works the same way.
\par We first need some definitions concerning $\mathsf{O}$-modules. We will freely use the notions of projective modules (see \cite[Chapter 1.2]{Lam}) and invertible fractional ideals (see \cite[p. 220]{Kaplansky}).
\begin{definition} Let $M$ be a finitely generated projective right $\mathsf{O}$-module. An element $v\in M$ is called a \emph{primitive vector} if the following two conditions are satisfied:
\begin{itemize}
\item $\Ann_\mathsf{O}(v)=0$
\item If $v=v' a$ for some $v'\in M, a\in \mathsf{O}$, then $a\in \mathsf{O}^\times$.
\end{itemize}
\end{definition}
Notice that the first condition follows from the second if $B$ is a division algebra. We are now ready to discuss Hermitian forms.
\begin{definition} A \emph{Hermitian $\mathsf{O}$-lattice} is a finitely generated projective right $\mathsf{O}$-module $M$ together with a pairing
$(\cdot,\cdot): M\times M\longrightarrow \mathsf{O} $ such that for all $v,v_1,v_2\in M, \lambda_1, \lambda_2\in \mathsf{O}$
\begin{itemize}
\item[i)] $(\cdot,\cdot)$ is linear in the second component:
$$(v,v_1 \lambda_1+v_2 \lambda_2)=(v,v_1) \lambda_1+(v,v_2)\lambda_2\, .$$
\item[ii)] $(\cdot,\cdot)$ is semilinear in the first component:
$$(v_1 \lambda_1+v_2 \lambda_2,v)=\overline{\lambda_1} (v_1,v) +\overline{\lambda_2}(v_2,v)\, .$$
\item[iii)] $(\cdot,\cdot)$ is hermitian symmetric: $(v_1,v_2)=\overline{(v_2,v_1)}\,.$
\end{itemize}
A pairing satisfying i) and ii) is called a sesquilinear pairing.
\end{definition}
\begin{lem} Let $M$ be a projective right $\mathsf{O}$-module. Then
$$M^\dagger=\left\lbrace \phi:M\longrightarrow \mathsf{O}|\,  \phi(v_1 \lambda_1+v_2 \lambda_2)=\overline{\lambda_1} \phi(v_1)+\overline{\lambda_2} \phi(v_2) \right\rbrace$$
is naturally a projective right $\mathsf{O}$-module. To give a sesquilinear pairing on $M$ is equivalent to giving a linear map $M\longrightarrow M^\dagger$.
\end{lem}
\begin{proof}
Clear.
\end{proof}
\begin{definition} We say that a sesquilinear pairing on $M$ is \emph{non-degenerate} (resp. \emph{perfect}) if the induced map $M\longrightarrow M^\dagger$ is injective (resp. bijective).
\end{definition}
\begin{lem}\label{PerfPair} Assume now that every invertible fractional left $\mathsf{O}$-ideal is free. Let $M$ be a projective right $\mathsf{O}$-module such that $M\otimes_R K$ is a free $B$-module. Let further $(\cdot,\cdot)$ be a perfect sesquilinear pairing on $M$. Let $v\in M$ be a primitive vector. Then there exists a $w\in M$ such that $(w,v)=1$.
\end{lem}
\begin{proof}
Let $\phi$ be  the image of $v$ under the map $M\longrightarrow M^\dagger$ induced by the pairing $(\cdot,\cdot)$. Denote by $I\subseteq \mathsf{O}$ the image of $\phi$. Then $I$ is naturally a left ideal of $\mathsf{O}$ (notice that the anti-linearity of $\phi$ interchanges left and right.). We claim that $KI=B$. Indeed $KI$ is a left ideal of $B$. By Wedderburn's theorem $B$ is either a division algebra or a $2\times 2$-matrix algebra over $K$. If $B$ is a division algebra there is nothing to show. Assume now that $B=M_{2,2}(K)$ is a matrix algebra over $K$. Suppose by way of contradiction that $KI\neq B$. It is well known and can be shown using elementary linear algebra that every proper left ideal of a matrix algebra $M_{n,n}(K)$ over a field is of the form $M_{n,n}(K) A$ where $\rk A < n$. Therefore there is a matrix $A\in M_{2,2}(K)$ with $\rk A<2$ such that $KI=B A$. Choose a basis $e_1,\ldots, e_d$ of the free $B$-module $M\otimes_R K$. Consider the $B$-antilinear map $\phi_K: M\otimes K \longrightarrow B$ obtained by extension of scalars. Since $\im(\phi_K)=KI$, there exist $r_1,\ldots, r_d\in B$ such that $\phi_K(e_i)=r_i A$. This implies that $\phi_K$ is the map $M\otimes K \longrightarrow B,\, \sum_{i=1}^d e_i \lambda_i\mapsto \overline{\lambda_i} r_i A$. Hence there is a $\phi'\in M^\dagger\otimes K=(M\otimes K)^\dagger$ such that $\phi=\phi' A$. On the other hand clearly $M\otimes K \longrightarrow M^\dagger\otimes K$ is an isomorphism. Therefore there is an $v'\in M\otimes K$ such that $v=v' A$. But since $A$ is not of full rank there exists an $r\in B\setminus\{0 \}$ such that $A r=0$. By multiplying with an element in $R$ we can assume that $r\in \mathsf{O}$. Then one has $vr=0$ contradicting the primitivity of $v$. We thus conclude that $KI=B$.
\par By \cite[Theorem 7]{Kaplansky} this implies that $I$ is an invertible left ideal. Since we assumed that all invertible left ideals are principal, there exists an $a\in \mathsf{O}$ such that $I=\mathsf{O} a$. As $KI=B$ the element $a$ cannot be a zero-divisor. Therefore $N(a)=a \overline{a} \neq 0$. Now the map $\phi':M\longrightarrow \mathsf{O},\, w\mapsto \phi(w) \frac{\overline{a}}{N(a)}$ is well-defined and gives an element in $M^\dagger$ satisfying $\phi=\phi' a$. Because the pairing $(\cdot, \cdot)$ is perfect, the map $M\longrightarrow M^\dagger$ is an isomorphism. Denote by $v'$ the pre-image of $\phi'$. Then $v=v' a$ and by the primitivity of $v$ we conclude that $a\in \mathsf{O}^\times$. Therefore $I=\mathsf{O}$ and hence by the definition of $I$ there exists an element $w\in M$ with $(w,v)=1$.
\end{proof}
\section{Affine Group Schemes}
In this chapter we recall some well-known facts about the group scheme $\alpha_p$ and related group schemes we will use. Throughout this chapter $k$ will be a field of characteristic $p>0$ and $S$ a scheme over $k$.
\begin{definition} An \emph{$\alpha$-group scheme} of rank $r$ over $S$ is a finite flat group scheme $G$ over $S$, such that there exists an fppf-covering $T\rightarrow S$ such that $G\times_S T\cong \alpha_{p, T}^r$ as $T$-group schemes.
\end{definition}
\begin{thm}\label{AlphaThm} The functor
$$\left\{
\begin{tabular}{@{}l@{}}
$\alpha$-group schemes\\
 over $S$
\end{tabular}
\right\} \longrightarrow \left\{
\begin{tabular}{@{}l@{}}
Locally free sheaves\\
 of finite rank on $S$
\end{tabular}
\right\}$$
that sends $G$ to $\text{\emph{Lie}}(G)^{\vee}$ is a rank preserving anti-equivalence of categories.
\end{thm}
\begin{proof}
Indeed we claim that the inverse functor $$\left\{
\begin{tabular}{@{}l@{}}
Locally free sheaves\\
 of finite rank on $S$
\end{tabular}
\right\} \rightarrow \left\{
\begin{tabular}{@{}l@{}}
$\alpha$-group schemes\\
 over $S$
\end{tabular}
\right\}$$
is defined as follows: For a locally free sheaf of finite rank $\mathcal{F}$ on $S$ we define the following coherent sheaf of $\mathcal{O}_S$-Hopf algebras $\mathcal{A}=\oplus_{i=0}^{p-1} \text{Sym}^i(\mathcal{F})$. The multiplication on $\mathcal{A}$ is given by $\text{Sym}^i(\mathcal{F})\otimes \text{Sym}^j(\mathcal{F}) \rightarrow \text{Sym}^{i+j}(\mathcal{F})$ if $i+j<p$ and $\text{Sym}^i(\mathcal{F})\otimes \text{Sym}^j(\mathcal{F}) \rightarrow 0$ if $i+j \geqslant p$. The comultiplication is defined as follows: One can easily show that there is a unique algebra map $c:\mathcal{A}\rightarrow \mathcal{A}\otimes_{\mathcal{O}_S} \mathcal{A}$ such that $c_{|\mathcal{F}}$ equals $(\id,\id):\mathcal{F}\rightarrow (\mathcal{F}\otimes \mathcal{O}_S)\oplus( \mathcal{O}_S\otimes \mathcal{F})\subset \mathcal{A}\otimes \mathcal{A}$. The counit is the map $\mathcal{A}\rightarrow \mathcal{O}_S$ that projects onto the first summand. Then we define $\alpha(\mathcal{F})=\underline{\text{Spec}}_S(\mathcal{A})$. This defines an $\alpha$-group scheme of rank $\text{rk}(\mathcal{F})$.
\par It is clear that both maps are contravariant functors. \cite[expos\'e VII\textsubscript{A}, Theorem 7.4]{SGA3} implies that $\alpha(\cdot)$ is an anti-equivalence with inverse $\text{Lie}(\cdot)^\vee$.
\end{proof}
Notice that the previous theorem implies that $\alpha$-group schemes are isomorphic to $\alpha_p^r$ locally for the Zariski topology (instead of fppf).
\begin{definition}\label{alphagroups} For a locally free sheaf $\mathcal{F}$ we define the corresponding $\alpha$-group scheme $\alpha(\mathcal{F})$ as in the previous proof.
\end{definition}
The previous theorem can be used to compute an isomorphism $G\cong \alpha_p$ for a group scheme which is abstractly isomorphic to $\alpha_p$ but no isomorphism is given a priori.
\begin{lem}\label{alpha} Let $G$ be a finite group scheme over a field $k$ which is isomorphic to $\alpha_p$. Suppose the Hopf-Algebra of $G$ is given as $A_G=k[t]/(t^p)$ (but the comultiplication does not have to be $t\mapsto t\otimes 1 +1 \otimes t$).
\par Let $\omega=\sum_{i=0}^{p-1} a_i t^i dt$ be a non-zero invariant differential form. Then there is a unique isomorphism
$$\iota:G\longrightarrow \alpha_p=\Spec k[x]/(x^p)$$
with $\iota^*(dx)=\omega$. It is given by
$$x\mapsto \sum_{i=0}^{p-2} \frac{a_i t^{i+1}}{i+1}\, ,$$
i.e. the logarithm of $G$ which is defined as the integral of $\omega$.
\par Furthermore $a_{p-1}=0$.
\end{lem}
\begin{proof}
The existence and uniqueness of $\iota$ follows from Theorem \ref{AlphaThm}. The explicit formula is an elementary computation.
\end{proof}
We will also need a fact about extensions of group schemes.
\begin{thm}\label{ext} Let $G$ be a finite commutative group scheme over $k$ and
$$0\longrightarrow \mathbb{G}_m \longrightarrow E \longrightarrow G \longrightarrow 0$$
an extension in the category of commutative group schemes. Suppose $k$ is perfect and $G$ is unipotent. Then the extension is uniquely split.
\end{thm}
\begin{proof}
See \cite[expos\'e XVII, Th\'eor\`eme 6.1.1(C)]{SGA3}.
\end{proof}
\begin{rem} Notice that over a non-perfect field there is a non-split extension of $\alpha_p$ by $\mathbb{G}_m$.
\end{rem}
\section{Theta Groups}
In this section we are going to recall the facts about theta groups needed for chapter \ref{main}. Let $A/k$ be an Abelian variety and $\mathcal{L}$ a line bundle on $A$. The theta group of $\mathcal{L}$ is a group that naturally encodes the geometric interplay of $\mathcal{L}$ with the group structure of $A$. The theta group satisfies nice properties under pullback. Therefore it is a useful theoretical as well as practical tool for studying a situation where a polarization is descended along an isogeny as in Moret-Bailly's construction. For this reason Moret-Bailly also uses theta groups in \cite{Moret-Bailly}.
\par The theory of theta groups is due to Mumford and was developed in the case where $\text{char}(k) \nmid \deg(\mathcal{L})$ in \cite{EqDefAV}. The general case is explained in a letter from Mumford to T. Sekiguchi published in \cite{Sekiguchi}. Our exposition follows \cite{vdGMooAV}. Notice that Sekiguchi's definition of a Heisenberg group differs from the one used here.
\subsection{Definition}
Let $A/k$ be an Abelian variety and $\mathcal{L}$ a line bundle. The line bundle induces a homomorphism $\lambda_\mathcal{L}:A\longrightarrow A^\vee,\, x \mapsto \mathcal{L}\otimes t^*_x \mathcal{L}^{-1}$. We denote the kernel of $\lambda_\mathcal{L}$ by $K(\mathcal{L})$. By \cite[p. 57]{MumAV} the line bundle $\mathcal{L}$ is ample if and only if $\lambda$ is an isogeny and $\HH^0(A, \mathcal{L})\neq 0$.
\begin{definition}
Let $A/k$ be an Abelian variety and $\mathcal{L}$ a line bundle. Then we define the \emph{theta group of $\mathcal{L}$} via the functor of points
\begin{equation*}
\begin{aligned}
\mathcal{G}(\mathcal{L}): \{\text{Sch} /k\}\longrightarrow \{\text{Groups}\}\\
T\mapsto \Big\lbrace(x,\phi)| x\in A(T)\textrm{ and } & \phi: \mathcal{L}_T \stackrel{\sim}{\longrightarrow} t_x^* \mathcal{L}_T\\
& \text{ an isomorphism of line bundles on } A\times T\Big\rbrace\,. 
\end{aligned}
\end{equation*}
The group structure is defined by $(x_1,\phi_1)\cdot (x_2,\phi_2)=(x_1+x_2,t_{x_2}^*\phi_1\circ \phi_2)$.
\end{definition}
\begin{lem}\label{ThetaExt} The functor $\mathcal{G}(\mathcal{L})$ is represented by a group scheme over $k$. There is an exact sequence of group schemes
$$0 \longrightarrow \mathbb{G}_m\longrightarrow \mathcal{G}(\mathcal{L})\longrightarrow K(\mathcal{L}) \longrightarrow 0$$
where the map $\mathcal{G}(\mathcal{L})\rightarrow K(\mathcal{L})$ is given by $(x,\phi) \mapsto x$.
\end{lem}
\begin{proof}
See \cite[Lemma 8.2]{vdGMooAV}.
\end{proof}
\begin{propdef} If $\mathcal{L}$ is ample, then $\mathbb{G}_m$ is equal to the center of $\mathcal{G}(\mathcal{L})$. Therefore the commutator map will factor through a non-degenerate alternating pairing
$$e^\mathcal{L}:K(\mathcal{L})\times K(\mathcal{L}) \rightarrow \mathbb{G}_m$$
called the \emph{commutator pairing}.
\end{propdef}
\begin{proof}
See \cite[Corollary 8.20]{vdGMooAV}.
\end{proof}
\subsection{Theta Groups and Pullback}
\begin{prop}\label{LevelPull} Let $A$ be an Abelian variety over a field $k$ and $\mathcal{L}$ be a line bundle on $A$. Let $\mathcal{H}\subset K(\mathcal{L})$ be a closed subgroup scheme. Denote the quotient map $A \rightarrow B=A/\mathcal{H}$ by $\pi$. Then there is a natural bijection
$$\left\{
\begin{tabular}{@{}l@{}}
\emph{line bundles} \\
$\mathcal{L}_B$ \emph{on} $B$ \emph{such}\\
\emph{that} $\pi^* \mathcal{L}_B \cong \mathcal{L}$
\end{tabular}
\right\} \stackrel{\sim}{\longleftrightarrow} \left\{
\begin{tabular}{@{}l@{}}
\emph{splittings of the}\\
\emph{exact sequence in}\\
\emph{Lemma \ref{ThetaExt} over} $\mathcal{H}\,.$
\end{tabular}
\right\}\,. $$
\end{prop}
\begin{proof}
See \cite[Theorem 8.10]{vdGMooAV}.
\end{proof}
\begin{cor}
 Suppose further that $\mathcal{L}$ is ample and $K(\mathcal{L})$ is unipotent. Then a line bundle $\mathcal{L}_B$ on $B$ with $\pi^* \mathcal{L}_B \cong \mathcal{L}$ exists if and only if $\mathcal{H}\subset K(\mathcal{L})$ is isotropic under the commutator pairing on $K(\mathcal{L})$. Furthermore $\mathcal{L}_B$ is then unique.
\end{cor}
\begin{proof}
Follows easily from Theorem \ref{ext}.
\end{proof}
\begin{definition} A level subgroup $\tilde{\mathcal{H}}\subset \mathcal{G}(\mathcal{L})$ is a subgroup scheme such that $\mathbb{G}_m \cap \tilde{\mathcal{H}}=\{1\}$.
\end{definition}
A level subgroup will map isomorphically onto its image under $\mathcal{G}(\mathcal{L}) \rightarrow K(\mathcal{L})$. Let us denote the image by $\mathcal{H}$. By Proposition \ref{LevelPull} the level subgroup $\tilde{\mathcal{H}}$ will determine a descended line bundle on $A/\mathcal{H}$.
\begin{propdef}
A subgroup $\tilde{\mathcal{H}}\subset \mathcal{G}(\mathcal{L})$ is called a maximal level subgroup if it is a level subgroup satisfying one (and hence both) of the following two equivalent properties:
\begin{itemize}
\item[i)] $\sharp(\tilde{\mathcal{H}})^2=\sharp(K(\mathcal{L}))$. Here and in the rest of the article $\sharp(\cdot)$ denotes the rank of a finite group scheme.
\item[ii)] Denote by $\mathcal{H}$ the image of $\tilde{\mathcal{H}}$ under $\mathcal{G}(\mathcal{L})\rightarrow K(\mathcal{L})$. Then $\mathcal{H}^\perp=\mathcal{H}$. The orthogonal complement should be taken with respect to the commutator pairing on $K(\mathcal{L})$.
\end{itemize}
If furthermore $k$ is algebraically closed then i) and ii) are equivalent to
\begin{itemize}
\item[iii)] $\tilde{\mathcal{H}}$ is maximal in the set of level subgroups ordered by inclusion.
\end{itemize}
\end{propdef}
\begin{proof}
See \cite[Lemma 8.22]{vdGMooAV}.
\end{proof}
For our applications it will be useful to translate the Proposition \ref{LevelPull} into the language of divisors. This leads to the following definition:
\begin{definition} Let $A/k$ be an Abelian variety and $\mathcal{H}\subset A$ be a subgroup scheme. A closed subscheme $Z\subset A$ is called $\mathcal{H}$-invariant if the scheme theoretic preimages of $Z$ under the two maps
$$\text{pr}_2: \mathcal{H}\times A \rightarrow A,\, m: \mathcal{H}\times A \rightarrow A,\, (h,a)\mapsto h+a$$
agree.
\end{definition}
For the proof of the next proposition we need a technical definition taken from \cite[p. 104]{MumAV}. The reader who does not want to study this proof can skip the definition. It will not be used in other parts of the paper.
\begin{definition}
Let $A/k$ be an Abelian variety, $\mathcal{H}\subset A$ be a subgroup scheme and $\mathcal{F}$ a coherent sheaf on $A$. A \emph{lift of the action} $m: \mathcal{H}\times A \rightarrow A$ to $\mathcal{F}$ is an isomorphism $\lambda:\text{pr}_2^*(\mathcal{F}) \stackrel{\sim}{\rightarrow} m^*(\mathcal{F})$ of sheaves on $\mathcal{H} \times A$ such that the following diagram of sheaves on $\mathcal{H}\times \mathcal{H} \times A$ commutes:
$$\begin{tikzcd}
\text{pr}_3^*(\mathcal{F}) \arrow{rr}{\pr_{23}^*(\lambda)}\arrow{dr}[swap]{(m_{12}\times\id_A)^*(\lambda)} && \xi^*(\mathcal{F})\arrow{dl}{(\id_{\mathcal{H}}\times m_{23})^*(\lambda)} \\
& \eta^*(\mathcal{F})
\end{tikzcd}$$
where $\xi=\pr_{23}\circ\, m,\, \eta=(m_{12}\times\id_A)\circ m=(\id_{\mathcal{H}}\times m_{23})\circ m $.
\end{definition}
We can now state the translation of Proposition \ref{LevelPull} into the language of divisors.
\begin{prop}\label{LevelPullDiv}  Let $A/k$ be an Abelian variety and $\mathcal{H}\subset A$ be a finite subgroup scheme. Suppose an effective divisor $D\subset A$ is $\mathcal{H}$-invariant. Denote by $\mathcal{L}=\mathcal{O}(D)$ and $\pi: A \rightarrow B=A/\mathcal{H}$ the natural quotient map.
\begin{itemize}
\item $\mathcal{H}\subset K(\mathcal{L})$ and there is a natural homomorphism $\mathcal{H}\rightarrow \mathcal{G}(\mathcal{L})$ splitting the exact sequence in Lemma \ref{ThetaExt} over $\mathcal{H}$.
\item $\mathcal{M}=\mathcal{O}(\pi(D))$ is a line bundle satisfying $\mathcal{L}\cong \pi^* \mathcal{M}$. Under the bijection in Proposition \ref{LevelPull} this corresponds to the splitting $\mathcal{H}\rightarrow \mathcal{G}(\mathcal{L})$ above.
\end{itemize}
\end{prop}
\begin{proof}
This is well-known and similar ideas are contained in \cite{EqDefAV}. As we could not find a proof in the literature we will indicate it here. The reader can skip the proof as it is quite technical and the techniques will not be used in the rest of the paper. Indeed consider the two maps $\text{pr}_2, m: \mathcal{H} \times A \rightarrow A$. Then both maps are flat: Indeed flatness of $\text{pr}_2$ is clear and flatness of $m$ follows from the commutative diagram
$$\begin{tikzcd} \mathcal{H} \times A \arrow[rr, "{(h,a) \mapsto (h,h+a)}"', "\sim" ] \arrow[dr, "m"']&& \mathcal{H} \times A \arrow{ld}{\text{pr}_2}\\
&A
\end{tikzcd}$$
Flatness implies that the natural map $\text{pr}_2^* \mathcal{O}(D) \rightarrow \mathcal{O}(\text{pr}_2^{-1}(D))$ is an isomorphism and similarly $m^* \mathcal{O}(D)\cong \mathcal{O}(m^{-1}(D))$. Because we assumed $D$ to be $\mathcal{H}$-invariant, one has $\text{pr}_2^{-1}(D)=m^{-1}(D)$. Therefore we obtain an isomorphism $\lambda:\text{pr}_2^* \mathcal{O}(D) \stackrel{\sim}{\rightarrow} m^* \mathcal{O}(D)$. By the definition of the theta group, this isomorphism determines a map of schemes $s:\mathcal{H}\rightarrow \mathcal{G}(\mathcal{L})$. We check that $s$ is a group homomorphism. Indeed this means that the two maps $s\circ m_{\mathcal{H}},\, m_{\mathcal{G}(\mathcal{L})\circ (s\times s)} : \mathcal{H}\times \mathcal{H} \rightarrow \mathcal{G}(\mathcal{L})$ are equal. Considered as an element in $\mathcal{G}(\mathcal{L})(\mathcal{H} \times \mathcal{H})$ the first map corresponds to the isomorphism $(m_{1,2}\circ \id_A)^* \lambda$. The other map is the product of the elements $s\circ \pr_1,\, s\circ \pr_2$ in $\mathcal{G}(\mathcal{L})(\mathcal{H} \times \mathcal{H})$. By the definition of the group law on the theta group this product corresponds to the isomorphism $(\id_{\mathcal{H}}\times m_{23})^*(\lambda)\circ \pr_{23}^*(\lambda)$. We are therefore reduced to proving that the following diagram of sheaves on $\mathcal{H}\times \mathcal{H} \times A$ commutes:
$$\begin{tikzcd}
\text{pr}_3^*(\mathcal{O}(D)) \arrow{rr}{\pr_{23}^*(\lambda)}\arrow{dr}[swap]{(m_{12}\times\id_A)^*(\lambda)} && \xi^*(\mathcal{O}(D))\arrow{dl}{(\id_{\mathcal{H}}\times m_{23})^*(\lambda)} \\
& \eta^*(\mathcal{O}(D))
\end{tikzcd}$$
where $\xi=\pr_{23}\circ m,\, \eta=(m_{12}\times\id_A)\circ m=(\id_{\mathcal{H}}\times m_{23})\circ m$. Indeed this is true because $\pr_3^{-1}(D)=\xi^{-1}(D)=\eta^{-1}(D)$ and all the maps are induced from the equality of these divisors. Therefore $s$ must be a group homomorphism.
\par It is clear that the composition $\mathcal{H} \rightarrow \mathcal{G}(\mathcal{L})\rightarrow K(\mathcal{L}) \subset A$ is the natural inclusion. Thus $\mathcal{H}\subset K(\mathcal{L})$ and $\mathcal{H}\rightarrow \mathcal{G}(\mathcal{L})$ is a splitting of the exact sequence in Lemma \ref{ThetaExt} over $\mathcal{H}$.
\par We will now prove the second part of the proposition. It is clear that we can instead consider $\mathcal{L}=\mathcal{O}(-D)$. In the first part of the proof we constructed a group homomorphism $\mathcal{H}\rightarrow \mathcal{G}(\mathcal{L})$ splitting the exact sequence in Lemma \ref{ThetaExt} over $\mathcal{H}$. Then Proposition \ref{LevelPull} gives a line bundle $\mathcal{M}$ on $B$ such that $\mathcal{L}\cong \pi^* \mathcal{M}$. We have to show that $\mathcal{M}=\mathcal{O}(-\pi(D))$. Indeed by the proof of Proposition \ref{LevelPull} (for which we refer to \cite[Theorem 8.10]{vdGMooAV}) the group homomorphism $\mathcal{H}\rightarrow \mathcal{G}(\mathcal{L})$ determines a lift of the multiplication action of $\mathcal{H}$ to $\mathcal{L}$. It is not difficult to check that this lift of the action is given by the isomorphism $\lambda$ from above. Furthermore the natural inclusion $\iota:\mathcal{O}(-D)\rightarrow \mathcal{O}_A$ is equivariant for this action, since this is equivalent to the commutativity of
$$\begin{tikzcd}
\pr_2^* \mathcal{O}(-D) \arrow{r}{\lambda} \arrow{d}{\pr_2^*(\iota)} & m^* \mathcal{O}(-D)\arrow{d}{m^*(\iota)} \\
\pr_2^*\mathcal{O}_A \arrow[equal]{r} & m^* \mathcal{O}_A
\end{tikzcd}$$
where we use the equality $\mathcal{O}_{\mathcal{H}\times A}=\pr_2^*\mathcal{O}_A= m^* \mathcal{O}_A$. Now by \cite[p. 104, Theorem 1]{MumAV} $\mathcal{M}=\pi_* (\mathcal{L})^{\mathcal{H}}$ where $\pi_* (\mathcal{L})^{\mathcal{H}}\subseteq \pi_* (\mathcal{L})$ is the subsheaf of $\mathcal{H}$-invariant sections as defined in the proof of the quoted theorem. But since the inclusion $\mathcal{O}(-D)\subset \mathcal{O}_A$ is $\mathcal{H}$-equivariant we have
$$\mathcal{M}=\pi_*(\mathcal{L})^{\mathcal{H}}=\pi_*(\mathcal{O}(-D))\cap \pi_*(\mathcal{O}_A)^{\mathcal{H}}=\pi_*(\mathcal{O}(-D))\cap \mathcal{O}_B\,.$$
Because $\pi$ is finite and thus affine, the latter sheaf equals the coherent sheaf of ideals whose vanishing set is the closed set $\pi(D)$ (this can be checked on affine opens). We conclude $\mathcal{M}=\mathcal{O}(-\pi(D))$. This finishes the proof of the proposition.
\end{proof}
\begin{prop}\label{Pull} Let $A$ be an Abelian variety over an algebraically closed or finite field $k$. Let $\mathcal{L}$ be a line bundle on $A$ such that $K(\mathcal{L})$ contains a non-zero Abelian subvariety $B\subseteq A$. Denote by $\pi: A\longrightarrow A/B$ the quotient map. Then there is a line bundle $\mathcal{L}_0$ on $A$ algebraically equivalent to $\mathcal{O}_A$ and a line bundle $\mathcal{L}_{A/B}$ on $A/B$ such that $\mathcal{L}\cong \mathcal{L}_0\otimes \pi^* \mathcal{L}_{A/B}$
\end{prop}
\begin{proof}
The case where $k$ is algebraically closed is given as an exercise in \cite[11.3]{vdGMooAV}. As there is not proof given there we shall include it here: Consider the restriction $\mathcal{L}_{|B}$. Since $B\subseteq K(\mathcal{L})$, the line bundle $\mathcal{L}_{|B}$ is algebraically equivalent to $\mathcal{O}_B$. Therefore it gives an element in $B^\vee(k)=\Pic^0(B)(k)$. Consider the exact sequence
$$0\longrightarrow (A/B)^\vee \longrightarrow A^\vee \longrightarrow B^\vee \longrightarrow 0\,.$$
Since $k$ is finite or algebraically closed the sequence
$$0\longrightarrow (A/B)^\vee(k) \longrightarrow A^\vee(k) \longrightarrow B^\vee(k) \longrightarrow 0$$
is exact by Lang's theorem or Hilbert's Nullstellensatz respectively. Therefore there is a line bundle $\mathcal{L}_0 \in A^\vee(k)$ such that $(\mathcal{L}_0^{-1}\otimes \mathcal{L})_{|B}\cong \mathcal{O}_B$. It remains to show that $\mathcal{L}_0^{-1}\otimes \mathcal{L}$ descends to $A/B$. Indeed after replacing $\mathcal{L}$ with $\mathcal{L}_0^{-1}\otimes \mathcal{L}$ we will now assume that $\mathcal{L}_{|B} \cong \mathcal{O}_B$. Consider the exact sequence
$$0\longrightarrow \mathbb{G}_m \longrightarrow \mathcal{G}(\mathcal{L})\longrightarrow K(\mathcal{L}) \longrightarrow 0$$
and push it out along $B\subseteq K(\mathcal{L})$: 
$$0\longrightarrow \mathbb{G}_m \longrightarrow \mathcal{G}(\mathcal{L})_{|B}\longrightarrow B \longrightarrow 0\,.$$
The resulting exact sequence is the sequence of $\mathcal{G}(\mathcal{L}_{|B})$ on $B$. This is the extension corresponding to $\mathcal{L}_{|B}$ under the identification $B^\vee(k)= \Ext^1(B,\mathbb{G}_m)$. Therefore the exact sequence splits. This implies that $B$ lifts to a level group in $\mathcal{G}(\mathcal{L})$. We conclude that $\mathcal{L}$ descends to $A/B$.
\end{proof}
\begin{definition}\label{DefHeis} Let $\mathcal{H}$ be a finite $k$-group scheme. We will define a group scheme $\mathcal{G}(\mathcal{H})$, the so called Heisenberg group associated to $\mathcal{H}$: As a scheme
$$\mathcal{G}(\mathcal{H})=\mathbb{G}_m \times \mathcal{H} \times \mathcal{H}^D$$
and the group structure is defined on $T$-valued points by
$$(r, h, \chi)\cdot(r', h', \chi')=(r r' \chi(h'), h+h', \chi+\chi')$$
where $\mathcal{H}^D$ denotes the Cartier dual of $\mathcal{H}$. We use the identification $\mathcal{H}^D(T)=\Hom_T(\mathcal{H}_T, \mathbb{G}_{m,T} )$.
\end{definition}
Notice that Heisenberg groups have very similar properties as the theta group of an ample line bundle. However such a theta group will not always be isomorphic to a Heisenberg group. On the other hand it will be isomorphic to a Heisenberg group if the degree of $\mathcal{L}$ is not divisible by the characteristic of the ground field. Another example will be the line bundle appearing in the construction of Moret-Bailly although in this case $\deg \mathcal{L}=p$. There is the following characterization for the theta group being a Heisenberg group:
\begin{lem}\label{Heis} Let $\mathcal{L}$ be an ample line bundle on an Abelian variety $A$. Suppose there are two maximal level subgroups $\tilde{\mathcal{H}}_1, \tilde{\mathcal{H}_2} \subset \mathcal{G}(\mathcal{L})$ with\\
$\tilde{\mathcal{H}}_1\cap  \tilde{\mathcal{H}_2}=\{1\}$ (scheme-theoretic intersection). Then $\tilde{\mathcal{H}}_1\cong \tilde{\mathcal{H}}_2^D$ and\\
 $\mathcal{G}(\mathcal{L})\cong \mathcal{G}(\tilde{\mathcal{H}}_1)$.
\end{lem}
\begin{proof}
Follows immediately from the proof of \cite[Lemma 8.24]{vdGMooAV}. The assumption that $k$ is algebraically closed is not needed for the part we use.
\end{proof}
\subsection{Representations of the Theta group}
In this section we recall the important properties of the action of $\mathcal{G}(\mathcal{L})$ on the global sections $\HH^0(A, \mathcal{L})$. First there is the following definition:
\begin{definition} Let $\mathcal{L}$ be a line bundle on an Abelian variety $A/k$. We say that a linear representation of $\mathcal{G}(\mathcal{L})$ has \emph{weight} $n\in \mathbb{Z}$ if the restriction to $\mathbb{G}_m$ is the character $z\mapsto z^n$.
\end{definition}
\begin{thm}\label{ThetaRep} Let $\mathcal{L}$ be an ample line bundle on an Abelian variety $A/k$. Then there is a natural absolutely irreducible weight 1 representation of the group scheme $\mathcal{G}(\mathcal{L})$ on $\HH^0(A, \mathcal{L})$.
\end{thm}
\begin{proof}
See \cite[p. 710]{Sekiguchi}.
\end{proof}
\begin{thm}\label{ThetaRepUniq} Let $\mathcal{L}$ be an ample line bundle on an Abelian variety $A/k$. Then there is a unique absolutely irreducible weight 1 representation $V$ of $\mathcal{G}(\mathcal{L})$. Its dimension is $\dim(V)=\sqrt{\sharp ( K(\mathcal{L}))}=\deg(\mathcal{L})$.
\end{thm}
\begin{proof}
The existence of such a representation follows from Theorem \ref{ThetaRep}. The uniqueness is proven in \cite[theorem 8.32]{vdGMooAV} in the case where $k$ is algebraically closed. We claim that the general case can be deduced using descent. \footnote{Alternatively it is not difficult to see that the cited proof goes through in the general case if we replace irreducible with absolutely irreducible everywhere.} Indeed suppose $V,W$ are two absolutely irreducible $\mathcal{G}(\mathcal{L})$-modules of weight $1$. Then $V\otimes \overline{k} ,W\otimes \overline{k}$ are irreducible $\mathcal{G}(\mathcal{L})_{\overline{k}}$-modules of weight $1$. By the quoted theorem in the case where the ground field is algebraically closed $V\otimes \overline{k}$ and $W\otimes \overline{k}$ are isomorphic. The next Lemma \ref{repdecent} implies that $V$ and $W$ are isomorphic as $\mathcal{G}(\mathcal{L})$-modules. This proves the claim. 
\end{proof}
\begin{lem}\label{repdecent} Let $k$ be a field and $G$ an affine group scheme over $k$ and $V,W$ be two irreducible $G$-modules. Suppose that $V\otimes \overline{k}$ and $W\otimes \overline{k}$ are isomorphic as $G_{\overline{k}}$-modules. Then $V$ and $W$ are isomorphic as $G$-modules.
\end{lem}
\begin{proof}
Consider the $k$-vector space $\Hom_G(V,W)$ of $G$-equivariant $k$-linear maps. By Schur's lemma it suffices to show that $\Hom_G(V,W)\neq \{0\}$.
\par Indeed $\Hom_G(V,W)$ equals the vector space of $k$-linear maps $\varphi:V\rightarrow W$ such that the following diagram commutes:
$$\begin{tikzcd}
V \arrow{r}{\varphi}\arrow{d} & W\arrow{d}\\
V\otimes A_G \arrow{r}{\varphi\otimes \id}& W\otimes A_G
\end{tikzcd}$$
where $A_G$ denotes the Hopf-algebra of $G$ and the vertical arrows are induced from the action of $G$ on $V$ (resp. $W$).
\par This description of $\Hom_G(V,W)$ implies that for any field extension $k'/k$ one has 
$$\Hom_{G_{k'}}(V\otimes k', W\otimes k')=\Hom_G(V,W)\otimes k'\,.$$
Since by assumption $V\otimes \overline{k}$ and $W\otimes \overline{k}$ are isomorphic as $G_{\overline{k}}$-modules we obtain $\Hom_{G_{\overline{k}}}(V\otimes \overline{k}, W\otimes \overline{k})\neq \{0\}$. Therefore $\Hom_G(V,W)\neq \{0\}$ and the lemma follows.
\end{proof}
The representation of the theta group can be used to characterize the image of the global sections under pullback.
\begin{prop}\label{ThetaRepPull}
Let $\pi:A\rightarrow B$ be a finite surjective map of Abelian varieties with kernel $\mathcal{H}$. Let further $\mathcal{M}$ be an ample line bundle on $B$. Then the image of
$$\pi^*: \HH^0(B, \mathcal{M}) \rightarrow \HH^0(A, \pi^* \mathcal{M})$$
is $\HH^0(A, \pi^* \mathcal{M})^{\tilde{\mathcal{H}}}$ where $\tilde{\mathcal{H}}\subset \mathcal{G}(\pi^* \mathcal{M})$ is the level subgroup over $\mathcal{H}$ coming from Proposition \ref{LevelPull}.
\end{prop}
\begin{proof}
Mumford uses this without proof in \cite[§1, Theorem 4]{EqDefAV}. The result follows from the proof of Proposition \ref{LevelPull}, which we do not discuss, and \cite[p. 104, Theorem 1]{MumAV}.
\end{proof}
We will also need a result about the action of the level group when we have a situation as in Proposition \ref{LevelPullDiv}.
\begin{prop}\label{ThetaRepPullDiv}
Let $A/k$ be an Abelian variety and $\mathcal{H}\subset A$ be a finite subgroup scheme. Suppose an ample effective divisor $D\subset A$ is $\mathcal{H}$-invariant.
\par Then the action of the level group $\tilde{\mathcal{H}}\subset \mathcal{G}(\mathcal{O}(D))$ coming from Proposition \ref{LevelPullDiv} on $\HH^0(A, \mathcal{O}(D))$ is given by the translation action of $\mathcal{H}$ on the function field $K(A)$.
\end{prop}
\begin{proof}
From the proof of Proposition \ref{LevelPullDiv} we get a lift of the translation action of $\mathcal{H}$ to the sheaf $\mathcal{O}(D)$. Furthermore we have seen there that the inclusion $\mathcal{O} \rightarrow \mathcal{O}(D)$ is equivariant for this action. By taking the stalk at the generic point of $\mathcal{O} \rightarrow \mathcal{O}(D)$ we see that the action on global sections induced by the action of $\mathcal{H}$ on the sheaf $\mathcal{O}(D)$ is given by translation of rational functions. On the other hand one can check using the definitions that the action of $\mathcal{H}$ on global sections induced from the action on the sheaf $\mathcal{O}(D)$ coincides with the restriction of the representation of the theta group to the level group $\tilde{\mathcal{H}}$.
\end{proof}
\section{The Construction of Moret-Bailly}\label{main}
Let $k$ be an algebraically closed field of characteristic $p>2$. Fix a supersingular elliptic curve $E/k$ defined over $\mathbb{F}_p$ such that the relative Frobenius satisfies $F^2+p=0$. All schemes in this section will be $k$-schemes to allow ourselves to apply theorems for schemes over algebraically closed fields. However as we see later everything will be defined over $\mathbb{F}_{p^2}$.
\subsection{Notation}
We shall fix the following notation for the rest of the article.
\begin{itemize}
\item $E$ is a supersingular elliptic curve defined over $\mathbb{F}_p$ such that the relative Frobenius satisfies $F^2+p=0$.
\item $\mathsf{O}=\text{End}(E)$ and $B=\mathsf{O}\otimes \mathbb{Q}$. Notice that the conditions on $E$ imply that all the endomorphism of $E$ are defined over $\mathbb{F}_{p^2}$.
\item There is an additive bijection
$$\gamma: \text{Mat}_{m,n}(\mathsf{O})\stackrel{\sim}{\rightarrow}\text{Hom}(E^n,E^m) $$
$$ \Phi=(\varphi_{ij})\mapsto \left(\psi:E^n\rightarrow E^m,\, (P_1,\ldots, P_n)\mapsto \left(\sum_{j=1}^n \varphi_{1j}(P_j)), \ldots, \sum_{j=1}^n \varphi_{mj}(P_j)\right)\right)\,.$$
This bijection turns matrix multiplication into composition of maps. Notice that we use the convention that maps should be composed from right to left.
\item $\mu: E^2\rightarrow (E^2)^\vee$ denotes the natural product polarization.
\item The map
$$f\mapsto (\mu\circ \gamma(f):E^2\rightarrow (E^2)^\vee)$$
is a bijection from the set of positive definite hermitian matrices in $M_{2,2}(\mathsf{O})$ to the set of polarizations on $E^2$.
\item Let $A$ be an Abelian variety. By an elliptic curve $E_1\subset A$ we mean a closed subscheme which is a smooth curve of genus 1 with a distinguished rational point, but not necessarily the origin. Thus $E_1$ does not have to be an Abelian subvariety.
\end{itemize}
\subsection{Results of Moret-Bailly}\label{SectionResultsMB}
In this section we recall the construction of the families of Moret-Bailly. All the results in this section are taken from \cite{Moret-Bailly}. Thus they are just stated here without proof for the convenience of the reader and to fix the notation.
\par Consider the $\alpha$-group scheme $\alpha(\mathcal{O}(1))$ over $\mathbb{P}^1$ (see Definition \ref{alphagroups} for the notation). The map $\mathcal{O}^2\rightarrow \mathcal{O}(1)$ induces a map of $\mathbb{P}^1$-group schemes $\alpha(\mathcal{O}(1))\rightarrow \alpha_{p, \mathbb{P}^1}^2$. After choosing an isomorphism $E[F]\cong \alpha_p$ there is a composed map $\iota: \alpha(\mathcal{O}(1))\rightarrow E^2_{\mathbb{P}^1}$. The map will depend on the choice of the isomorphism $E[F]\cong \alpha_p$, but its image will not. By \cite[expos\'e 5 Th\'eor\`eme 4.1]{SGA3} the quotient $E^2_{\mathbb{P}^1}/\iota(\alpha(\mathcal{O}(1)))$ exists. It is an Abelian scheme over $\mathbb{P}^1$.
\begin{definition} We define $\mathcal{A}=E^2_{\mathbb{P}^1}/\iota(\alpha(\mathcal{O}(1)))$. 
\end{definition}
Moret-Bailly proved that one can put a principal polarization on $\mathcal{A}$:
\begin{lem}\label{MB} Let $\eta$ be a polarization on $E^2$ such that $\ker(\eta)=E^2[F]$. Then $\eta_{\mathbb{P}^1}$ descends to a principal polarization $\lambda$ on the Abelian scheme $\mathcal{A}\rightarrow \mathbb{P}^1$.
\end{lem}
\begin{proof}
For the proof we refer to \cite[§1.4]{Moret-Bailly}.
\end{proof}
The next proposition explains the properties of the polarization $\lambda$ on the Abelian scheme $\mathcal{A}\rightarrow \mathbb{P}^1$:
\begin{thm}\label{MBPol} There exists a symmetric divisor $\mathfrak{C}$ on $\mathcal{A}$ that is relatively ample and flat 
over $\mathbb{P}^1$ and induces the polarization $\lambda$. Then we have the following properties:
\begin{itemize}
\item The generic fiber of $\mathfrak{C} \rightarrow \mathbb{P}^1$ is a smooth geometrically connected curve of genus 2.
\item $5p-5$ fibers of $\mathfrak{C} \rightarrow \mathbb{P}^1$ are of the form $E_1\cup E_2$ where $E_1,E_2$ are elliptic curves intersecting transversely in a point contained in $\mathcal{A}[2]$. We are going to refer to these fibers as the reducible fibers. By \cite[formula 3.3]{KatsuraOort} the reducible fibers are fibers over $\mathbb{F}_{p^2}$-valued points of $\mathbb{P}^1$ (necessary but not sufficient condition).
\end{itemize}
Denote by $\mathfrak{D}$ the preimage of $\mathfrak{C}$ under $E^2_{\mathbb{P}^1}\rightarrow \mathcal{A}$. Then $\mathfrak{D}$ is a symmetric divisor on $E^2_{\mathbb{P}^1}$ that is relatively ample and flat over $\mathbb{P}^1$ and induces the polarization $\eta_{\mathbb{P}^1}$. Furthermore we have the following properties:
\begin{itemize}
\item The generic fiber of $\mathfrak{D} \rightarrow \mathbb{P}^1$ is a singular geometrically integral curve of geometric genus 2 and arithmetic genus $p+1$.
\item $5p-5$ fibers of $\mathfrak{D} \rightarrow \mathbb{P}^1$ are of the form $E_1\cup E_2$ where $E_1,E_2$ are elliptic curves intersecting with multiplicity $p$ in one of $10$ points contained in $E^2[2]$.
\end{itemize}
\end{thm}
\begin{proof}
See \cite[§2.1]{Moret-Bailly}.
\end{proof}
Notice that $\mathfrak{C}$ is not unique but only unique up to translating by a $2$-torsion point.

\subsection{Reducible Fibers I: Classification of reducible fibers}
\begin{prop}\label{decomp}
Let $f\in M_{2,2}(\mathsf{O})$ be a hermitian positive definite matrix inducing a polarization $\eta$ on $E^2$ with kernel $E^2[F]$. Then there exists an ample divisor $D=E_1+E_2$ inducing the polarization $\eta$ where $E_1,E_2$ are elliptic curves intersecting in the origin with multiplicity $p$. To give such a decomposition is equivalent to writing $f=f_1+f_2$ where $f_1,f_2\in M_{2,2}(\mathsf{O})$ are hermitian rank 1 matrices with integral entries.
\par Furthermore if we are given $f_1,f_2$ as above then $E_i=\ker(\gamma(f_i))$. In particular the $E_i$ and $D$ are defined over $\mathbb{F}_{p^2}$.
\end{prop} 
\begin{proof}
The existence of an ample divisor $D$ of the form $D=E_1+E_2$ with above properties follows from Theorem \ref{MBPol}. We prove now that to give such a divisor is equivalent to writing $f=f_1+f_2$ with $f_1,f_2\in M_{2,2}(\mathsf{O})$ hermitian of rank $1$. Indeed consider the map
$$r:\text{NS}(E^2)\longrightarrow \End(E^2),\, D \mapsto \mu^{-1} \circ \lambda_{\mathcal{O}(D)}$$
where $\mu$ denotes the natural product polarization on $E^2$. By \cite[p. 174, Theorem 2 and p. 214, Theorem 3]{MumAV} $r$ is an injective homomorphism and $\text{im}(r)$ is the subgroup of endomorphisms fixed by the Rosati involution induced by $\mu$. Under the identification $\End(E^2)=M_{2,2}(\mathsf{O})$ the Rosati involution is given by the hermitian transpose. The desired bijection is defined by $f_i=r(E_i),\, i=1,2$.
\par We prove that this bijection has the right properties. Indeed suppose we are given an ample divisor $D=E_1+E_2$ inducing the polarization $\eta$. We have to show that $f_i=r(E_i)$ have rank 1 for $i=1,2$. To show this notice first that clearly $E_i\subseteq \ker(\gamma(f_i)),\, i=1,2$. Furthermore by Poincar\'e's reducibility theorem $E^2$ is isogenous to $E_1\times E_2$ and therefore $E_1, E_2$ must be supersingular. Therefore there is an isogeny $E\longrightarrow E_i$. The composed map $E\rightarrow E_i \hookrightarrow E^2$ corresponds to a $v_i\in \mathsf{O}^2\setminus \{0\}$. Now $f_i v_i=0$ proving that $\rk f_i \leqslant 1$. Since certainly $f_i\neq 0$ we deduce $\rk f_i=1$.
\par Conversely suppose we are given $f_1,f_2\in M_{2,2}(\mathsf{O})$ hermitian of rank 1 such that $f=f_1+f_2$. Consider $D_i=r^{-1}(f_i),\, i=1,2$. We have to show that the $D_i$ are algebraically equivalent to elliptic curves with the required properties. Indeed as $\rk f_i=1$, there exist $v_i\in \mathsf{O}^2\setminus \{0\}$ such that $f_i v_i=0$. The $v_i$ correspond to maps $E\longrightarrow E^2$ whose image $E_i$ is an elliptic curve satisfying $E_i\subset \ker(\gamma(f_i))$. It follows from Proposition \ref{Pull} that $D_i$ is algebraically equivalent to $n_i E_i$ for some $n_i\in \mathbb{Z}$. On the other hand\\
$D.D=2\deg(D)=2\sqrt{\deg \eta}=2p=2n_1 n_2 E_1.E_2$. Therefore
$$(1,1,p),(1,p,1),(p,1,1),(-1,-1,p),(-1,-p,1),(-p,-1,1)$$
are all the possibilities for $(n_1,n_2, E_1.E_2)$. The last three possibilities are excluded by ampleness of $D$. The 2nd and 3rd would imply $E_1[p]\subseteq \ker(\gamma(f_1))\cap \ker(\gamma(f_2))\subseteq \ker(\eta)$ respectively $E_2[p]\subseteq \ker(\eta)$ contradicting $\ker(\eta)=E^2[F]$. This shows that $D=E_1+E_2$ and $E_1. E_2 = p$.
\par We will now prove that $E_i=\ker(\gamma(f_i))$. Indeed denote by $\mathcal{L}_i=\mathcal{O}(E_i)$ the line bundle corresponding to $E_i$. Since $f_i=\mu^{-1}\circ \lambda_{\mathcal{L}_i}$, the claim is equivalent to $K(\mathcal{L}_i)=E_i$. Thus we will compute $K(\mathcal{L}_i)$. Indeed clearly $E_i\subset K(\mathcal{L}_i)$ and $K(\mathcal{L}_i)\neq E^2$ because $E_i$ is not algebraically equivalent to $0$. Therefore $G=K(\mathcal{L}_i)/E_i$ is finite. On the other hand if $\pi: E^2 \rightarrow E^2/E_i$ denotes the natural quotient map, then $\mathcal{L}_i=\pi^* \mathcal{O}(0)$. By Proposition \ref{LevelPull} $E_i$ lifts to a level subgroup in $\mathcal{G}(\mathcal{L}_i)$. We compute the centralizer $\mathfrak{C}(E_i)$ in $\mathcal{G}(\mathcal{L}_i)$. For this we note that the commutator pairing factors through $E_i\times G \rightarrow \mathbb{G}_m$. This induces a group homomorphism $E_i \rightarrow G^D$. But $E_i$ is proper and geometrically connected and $G^D$ is finite. Therefore $E_i \rightarrow G^D$ is trivial. We conclude $\mathfrak{C}(E_i)=\mathcal{G}(\mathcal{L}_i)$. Using \cite[Proposition 8.15]{vdGMooAV} one obtains $\mathbb{G}_m=\mathcal{G}(\mathcal{O}(0_{E^2/E_i}))=\mathfrak{C}(E_i)/E_i$. This implies $G=0$ and we have proven $E_i=\ker(\gamma(f_i))$.
\end{proof}
\begin{cor}\label{RedQuot}
In the notation of the previous proposition, the scheme theoretic intersection $\mathcal{H}=E_1 \cap E_2 \subset E^2$ is a closed subgroup scheme isomorphic to $\alpha_p$. The quotient $A=E^2/\mathcal{H}$ is isomorphic to $E_1^{(p)}\times E_2^{(p)}$. Also $\eta$ descends along the isogeny $E^2\rightarrow A$ to a principal polarization $\lambda$ on $A$. In particular $(A,\, \lambda)$ is a fiber of the Moret-Bailly family constructed from $E^2, \eta$. Furthermore $\lambda$ is the natural product polarization on $E_1^{(p)}\times E_2^{(p)}$.
\end{cor}
\begin{proof}
Choose a uniformizer $\mathfrak{z}\in \mathcal{O}_{E_1,0}$. Since $E_1,E_2$ intersect with multiplicity $p$, the intersection $\mathcal{H}=E_1\cap E_2$ is given by the image of the closed immersion $\Spec(\mathcal{O}_{E_1,0}/\mathfrak{z}^p) \hookrightarrow E_1$. Therefore $\mathcal{H}=E_1[F]$ as a closed subscheme. We conclude that $\mathcal{H}$ is a closed subgroup scheme isomorphic to $\alpha_p$.
\par Consider the quotient $A=E^2/\mathcal{H}$ and denote the quotient map by $\pi: E^2\rightarrow A$. Now because $\mathcal{H}\subset E_1$ and $\mathcal{H}\subset E_2$ is a subgroup scheme, the divisor $D$ is $\mathcal{H}$-invariant. Therefore by Proposition \ref{LevelPullDiv} $\mathcal{O}(D)$ descends along $\pi$ to the line bundle $\mathcal{O}(\pi(D))$. Since $D$ induces the polarization $\eta$, $\eta$ thus descends to a polarization $\lambda$ induced by $\pi(D)$ on $A$. Computing degrees yields that $\lambda$ is a principal polarization. On the other hand $\mathcal{O}(D)=\pi^* \mathcal{O}(\pi(D))$ implies $D=\pi^{-1}(\pi(D))$. Therefore the map $D \rightarrow \pi(D)$ is a finite radicial map of degree $p$ being the pull-back of such a map. Thus $\pi(D)=E_1^{(p)} + E_2^{(p)}$ and by a result of A. Weil $A\cong E_1^{(p)}\times E_2^{(p)}$ and $\lambda$ is the natural product polarization.
\end{proof}
\begin{cor}\label{decompMB}
In the notation of the previous proposition there is a bijective map
$$\left\{
\begin{tabular}{@{}l@{}}
\emph{Unordered pairs} $f_1,f_2\in M_{2,2}(\mathsf{O})$\\
\emph{hermitian}
\emph{of rank 1}\\
\emph{such that} $f=f_1+f_2$
\end{tabular}
\right\} \rightarrow \left\{
\begin{tabular}{@{}l@{}}
\emph{reducible fibers of}\\
\emph{the Moret-Bailly family}\\
  $\mathfrak{C}\rightarrow \mathbb{P}^1$ \emph{constructed from} $f$
\end{tabular}
\right\}\,. $$
\end{cor}
\begin{proof}
Proposition \ref{decomp} and Corollary \ref{RedQuot} give us such a map. We claim that the map is injective. Indeed suppose that the two pairs $\{f_1,f_2\}$ and $\{f_1', f_2'\}$ map to the same element in the target. Define $D$ to be the divisor $\ker(\gamma(f_1))+\ker(\gamma(f_2))$ on $E^2$ and similarly $D'=\ker(f_1')+\ker(f_2')$. Define further $\mathcal{H}=\ker(\gamma(f_1))\cap \ker(\gamma(f_2))$ and $\mathcal{H}'=\ker(f_1')\cap \ker(f_2')$. Then one has $\mathcal{H}=\mathcal{H}'$ as subgroup schemes because $\{f_1,f_2\}$ and $\{f_1', f_2'\}$ map to the same element in the target. Denote by $\pi: E^2\rightarrow E^2/\mathcal{H}$ the quotient map. Then $\pi(D), \, \pi(D')$ induce the same principal polarization. Therefore there exists a $2$-torsion point $P$ such that $\pi(D)=t_P (\pi(D'))$. But the two components of $\pi(D)$ (resp. $\pi(D')$) intersect only at $0$. Therefore one must have 
$P=0$. This implies $\pi(D)=\pi(D')$. But by Proposition \ref{LevelPullDiv} one has $D=\pi^{-1}(\pi(D))$ (resp. $D'=\pi^{-1}(\pi(D)')$) and thus we conclude $D=D'$. Proposition \ref{decomp} gives $\{f_1,f_2\}=\{f_1',f_2'\}$ as unordered pairs. This proves the injectivity. We prove the surjectivity. Indeed let $x\in \mathbb{P}^1$ be a point such that $\mathfrak{C}_x$ is reducible. Consider the divisor $\mathfrak{D}_x$ on $E^2$. By Theorem \ref{MBPol} $\mathfrak{D}_x$ is the union of two elliptic curves intersecting at a $2$-torsion point $P\in E^2[2]$. Define $D=t_P(\mathfrak{D}_x)$. Then $D$ is an ample divisor inducing the polarization $\eta$ and the union of two elliptic curves intersecting at the origin. Proposition \ref{decomp} gives the desired preimage $\{f_1,f_2\}$ under our map.
\end{proof}
We further investigate these decompositions $f=f_1+f_2$ from the proposition.
\begin{prop}\label{decompLinAlg} Let $f\in M_{2,2}(\mathsf{O})$ be a hermitian positive definite matrix inducing a polarization $\eta$ on $E^2$ with kernel $E^2[F]$. Consider the hermitian lattice $L=(\mathsf{O}^2,f)$. To write $f=f_1+f_2$ where $f_1,f_2\in M_{2,2}(\mathsf{O})$ are hermitian rank 1 matrices is equivalent to giving two sublattices $I_1, I_2\subseteq L$ with the following properties i)-iii). If $I_1,I_2$ satisfy i) and ii), then condition iii) is equivalent to iii'):
\begin{itemize}
\item[i)] $I_1,I_2$ are invertible right-$\mathsf{O}$-modules.
\item[ii)] $I_1\perp I_2$ with respect to $f$.
\item[iii)] $L/(I_1\oplus I_2) \cong \mathsf{O}/F$.
\item[iii')] $I_2$ is primitive and for every $v\in I_1$ one has $v^\dagger f v=p [I_1:v \mathsf{O}]$.
\end{itemize}
The bijection is defined as follows: Given $f_1,f_2$ we define $I_i=\ker(f_i),\, i=1,2$ considered as an $\mathsf{O}$-submodule of $\mathsf{O}^2$.
\par Conversely given $I_1,I_2$ there is a direct sum decomposition
$$L \otimes B=(I_1 \otimes B) \oplus (I_2 \otimes B)\,.$$
Denote by $P_i\in M_{2,2}(B)$ the matrices of the projection maps onto the two factors respectively. Then  one defines $f_1=P_2^\dagger f P_2,\, f_2=P_1^\dagger f P_1$.
\end{prop}
\begin{proof}
``$\Rightarrow$'' Suppose we have a decomposition $f=f_1+f_2$ where $f_i\in M_{2,2}(\mathsf{O})$ are hermitian rank 1 matrices with integral entries. Then we define $I_i=\ker(f_i)$ considered as an $\mathsf{O}$-submodule of $\mathsf{O}^2$. It is clear that property i) holds. For showing ii) we take $v_1\in I_1,v_2\in I_2$ arbitrary. Then $v_1^\dagger f v_2= v_1^\dagger f_1 v_2+v_1^\dagger f_2 v_2=v_2^\dagger f_1 v_1=0$ using the definition of the $I_i$.
\par We turn to property iii): Indeed i) and ii) imply the direct sum decomposition
$$L \otimes B=(I_1 \otimes B) \oplus (I_2 \otimes B)\,.$$
Denote by $P_i\in M_{2,2}(B)$ the matrices of the projection maps onto the two factors respectively. Then one has the identities $f_i P_i=0,\, P_1+P_2=\mathds{1}_2$. From this we deduce $f_1=f_1(P_1+P_2)=f_1 P_2$ and similarly $f_2=f_2 P_1$. This implies $f P_2=f_1 P_2 + f_2 P_2=f_1$ and similarly $f P_1 =f_2$. We will now deduce iii) by considering the localization at all prime numbers $l$. Indeed notice first that $I_1\subset L$ is a primitive submodule, i.e., $L/I_i$ is torsion free. This implies that also $I_{1,(l)} \subset L_{(l)}$ is primitive. Furthermore since $\mathbb{Z}_{(l)}$ is a DVR, $I_{1,(l)}$ is free, generated by $v_1\in I_{1,(l)}$, say. Assume first that $l\neq p$. Then since $\ker(\eta)=E^2[F]$ and $l\neq p$, the pairing on $L_{(l)}$ is perfect. Thus by Lemma \ref{PerfPair} there exists an $w\in L_{(l)}$ such that $w^\dagger f v_1=1$. 
On the other hand $P_1$ is the orthogonal projection onto the subspace generated by $v_1$ and thus by basic linear algebra $P_1= \frac{v_1 v_1^\dagger f}{v_1^\dagger f v_1}$. We now compute $w^\dagger f_2 w = w^\dagger f P_1 w=\frac{w^\dagger f v_1 v_1^\dagger f w}{v_1^\dagger f v_1}=\frac{1}{v_1^\dagger f v_1}\in \mathbb{Z}_{(l)}$. This implies that $v_1^\dagger f v_1\in  \mathbb{Z}_{(l)}^\times$. Therefore $P_1\in M_{2,2}(\mathsf{O}_{(l)})$. Similarly one proves $P_2\in M_{2,2}(\mathsf{O}_{(l)})$. This implies $L_{(l)}=I_{1,(l)}\oplus I_{2,(l)}$. It remains to consider the case $l=p$. Then the pairing given by $f$ induces a map $L\longrightarrow L^\dagger$ with image $L^\dagger F$. The same argument as in Lemma \ref{PerfPair} shows the existence of a $w\in L$ such that $w^\dagger f v_1=F$. We now compute $w^\dagger f_2 w = w^\dagger f P_1 w=\frac{w^\dagger f v_1 v_1^\dagger f w}{v_1^\dagger f v_1}=\frac{F \overline{F}}{v_1^\dagger f v_1}=\frac{p}{v_1^\dagger f v_1}\in \mathbb{Z}_{(p)}$. This implies that $P_1=\frac{v_1 v_1^\dagger f}{v_1^\dagger f v_1}\in \frac{F}{p} M_{2,2}(\mathsf{O}_{(p)})$ because $f\in F M_{2,2}(\mathsf{O})$. Similarly $P_2\in \frac{F}{p} M_{2,2}(\mathsf{O}_{(p)}) $. We conclude that $L_{(p)}/(I_{1,(p)}\oplus I_{2,(p)})$ is killed by $F$. This implies that $L_{(p)}/(I_{1,(p)}\oplus I_{2,(p)})$ is either $\mathsf{O}/F$ or $(\mathsf{O}/F)^2$. However the second case is impossible because of the primitivity of $I_{1,(p)}$. This proves iii).
\par ``$\Leftarrow$'': Conversely suppose we are given $I_1$ and $I_2$ satisfying i)-iii). There is a direct sum decomposition
$$L \otimes B=(I_1 \otimes B) \oplus (I_2 \otimes B)\,.$$
Denote by $P_i\in M_{2,2}(B)$ the matrices of the projection maps onto the two factors respectively. Then one defines $f_1=P_2^\dagger f P_2,\, f_2=P_1^\dagger f P_1$. By construction one has $f_1 P_1=f_2 P_2=0$. From these equalities and $P_1+P_2=\mathds{1}_2$ one easily infers $f=f_1+f_2,\, f_1=f P_2,\, f_2=f P_1$. It remains to show $f_1, f_2\in M_{2,2}(\mathsf{O})$. Indeed from ii) one gets $P_1,P_2\in \frac{F}{p} M_{2,2}(\mathsf{O})$. Hence one obtains $f_1=f P_2,\, f_2=f P_1\in M_{2,2}(\mathsf{O})$.
\par We show now that for $I_1, I_2$ satisfying conditions i)-iii) property iii') holds true. Indeed we have already explained in this proof why $I_2$ must be primitive. We prove now the formula $v^\dagger f v = p [I_1: v \mathsf{O}],\, \forall v \in I_1$ by comparing the $l$-adic valuations of both sides for all prime numbers. Since $I_{1,(l)}, I_{2,(l)}$ are free, there exist $v_1\in I_{1,(l)},\, v_2\in I_{2,(l)} $ such that $I_{1,(l)}=v_1 \mathsf{O}_{(l)}, I_{2,(l)}=v_2 \mathsf{O}_{(l)}$. Because $I_1\perp I_2$, the matrix representing the hermitian form on $I_{1,(l)}\oplus I_{2,(l)}$ induced by $f$ is of the form
$$\begin{pmatrix}
d_1 & 0\\
0 & d_2
\end{pmatrix}$$  with $d_1,d_2\in \mathbb{Z}_{(l)}$. Assume first that $l\neq p$. Then the hermitian pairing on $I_{1,(l)}\oplus I_{2,(l)}$ is perfect an thus $d_1,d_2 \in \mathbb{Z}_{(l)}$. This implies $\nu_l(v^\dagger f v) = \nu_l( [I_1: v \mathsf{O}]),\, \forall v \in I_1$. Assume now that $l=p$. Then as $F|f$ we must have $p| d_1, d_2$. On the other hand the map $I_{1,(p)}\oplus I_{2,(p)} \longrightarrow I_{1,(p)}^\dagger\oplus I_{2,(p)}^\dagger$ induced by $f$ has cokernel of length $4$ over the local ring $\mathsf{O}_{(p)}$ with maximal ideal $F \mathsf{O}_{(p)}=\mathsf{O}_{(p)} F$: This follows from $L/(I_1\oplus I_2)=\mathsf{O}/F$ and $\im (L \longrightarrow L^\dagger) = L^\dagger F$. This implies that $\nu_p(d_1 d_2)=\nu_p(F^4)=2$. Therefore $\nu_p(d_1)=\nu_p(d_2)=1$. We conclude $\nu_p(v^\dagger f v) = 1+\nu_p( [I_1: v \mathsf{O}]),\, \forall v \in I_1$. Since the sign of $v^\dagger f v$ is positive we have proven $v^\dagger f v = p [I_1: v \mathsf{O}],\, \forall v \in I_1$.
\par Next we prove that for $I_1, I_2$ satisfying conditions i), ii), iii') property iii) holds true. Indeed given $I_1, I_2$ we define $f_1, f_2$ by the formulae above. We claim that $f_1, f_2$ have integral coefficients. This is proven by localizing at all prime numbers $l$. Using the same notation as in the previous paragraph we have $f_1=\frac{f v_1 v_1^\dagger f}{v_1^\dagger f v_1}$. The matrix coefficients of numerator have $l$-adic valuation $\geqslant \nu_l(p)$, while the denominator has $\nu_l(v_1^\dagger f v_1)=\nu_l(d_1)=\nu_l(p)$. This implies that $f_1\in M_{2,2}(\mathsf{O})$. Therefore also $f_2=f-f_1\in M_{2,2}(\mathsf{O})$. Furthermore by construction $I_i \subseteq \ker(f_i)$. Equality holds because all the lattices are primitive sublattices. Therefore by the direction ``$\Rightarrow$'' we already proved, $I_1, I_2$ satisfy conditions i)-iii).
\end{proof}
\begin{cor}
Using the notation of the previous proposition, the isomorphism classes of the elliptic curves $E_i$ in Proposition \ref{decomp} are determined as follows: By a result of Deuring there is a bijection between isomorphism classes of invertible right-$\mathsf{O}$-modules and isomorphism classes of supersingular elliptic curves over $\overline{\mathbb{F}_{p}}$. Then $E_i$ is the elliptic curve corresponding to $I_i$ under this bijection.
\par In particular one has $E_i\cong E$ if and only if $I_i$ is free as a $\mathsf{O}$-module generated by some vector $v_i\in \mathsf{O}^2$. In this case the map $E_i\hookrightarrow E^2$ is the one given by the vector $v_i$.
\end{cor}
\begin{proof}
 We have to show that the bijection
$$\left\{
\begin{tabular}{@{}l@{}}
Invertible right\\
 $\mathsf{O}$-modules
\end{tabular}
\right\}_{/\cong} \longrightarrow \left\{
\begin{tabular}{@{}l@{}}
Supersingular elliptic\\
 curves over $\mathbb{F}_{p^2}$
\end{tabular}
\right\}_{/\cong} $$
maps $I_1$ to the elliptic curve $E_1$. Indeed choose $v\in I_1\setminus\{0\}$. Then $v\in \mathsf{O}^2$ determines a map $\phi: E\rightarrow E^2$. Since $v\in I_1=\ker(f_1)$, one has $\im(\phi)\subseteq E_1$. Since $v\neq 0$ the map $\phi:E\longrightarrow E_1$ is an isogeny. One checks that $\phi$ is the isogeny constructed from the invertible left $\mathsf{O}$-ideal
$$\{a\in B | v a \in I_1 \}^{-1} $$
in the construction in \cite[Section 42.2]{Voight}.
\end{proof}
\begin{exmp} Consider for example $E/\mathbb{F}_3$ with Weierstraß equation $y^2=x^3-x$. Choose a square root of -1 in $\mathbb{F}_9$ and denote it by $i$. Then $E$ has the endomorphism $x\mapsto -x, \, y\mapsto i y$ defined over $\mathbb{F}_9$ also denoted $i$ by abuse of notation. Therefore $E$ has CM by the order $\mathbb{Z}[i]$. Since $(3)$ is inert in $\mathbb{Q}(i)$ the curve $E$ is supersingular and
$$\End(E)\otimes \mathbb{Q} = \mathbb{Q}[i,F]/(iF+Fi)$$
is the endomorphism algebra. We can take $$f=\begin{pmatrix}
3 & (1+i)F\\
-(1+i)F& 3
\end{pmatrix}\,.$$
Then one decomposition as in the proposition is
$$f_1=\begin{pmatrix}
2 & (1+i)F\\
-(1+i)F& 3
\end{pmatrix},f_2=\begin{pmatrix}
1& 0\\
0 & 0
\end{pmatrix}\,.$$
One computes that $I_1=\begin{pmatrix}
-F\\ 1+i
\end{pmatrix} \mathsf{O}$
and 
$I_2=\begin{pmatrix}
0\\1
\end{pmatrix} \mathsf{O}$. Hence $$P_1=\begin{pmatrix}
1 & 0\\
\frac{(1+i)F}{3} & 0
\end{pmatrix},\, P_2=\begin{pmatrix}
0 & 0\\
\frac{-(1+i)F}{3} & 1
\end{pmatrix}\,.$$
One checks that the formulas $f_1=P_2^\dagger f P_2,\, f_2=P_1^\dagger f P_1$ are indeed satisfied.
\end{exmp}
The following algorithm computes for a given $f$ and an invertible right $\mathsf{O}$-module $I$ all the sublattices $I_1\oplus I_2$ with properties i)-iii) of the last proposition with $I_1\cong I$. As it turns out this computation reduces to a short vector problem.\\
\begin{algorithm}[H]\label{AlgDec}
\caption{Compute decompositions}
\DontPrintSemicolon
        \SetKwInOut{Input}{input}
        \SetKwInOut{Output}{output}

        \Input{An hermitian positive definite matrix $f\in M_{2,2}(\mathsf{O})$ inducing a polarization $\eta$ on $E^2$ with kernel $E^2[F]$. A fractional ideal $I\subset B$ with right order $\mathsf{O}$.}
        \Output{All decompositions $f=f_1+f_2$ as in Proposition \ref{decomp} such that $\ker(f_1)\cong I$}
        \Begin{
        Consider the hermitian lattice $(\mathsf{O}^2,f)$. By forgetting the $\mathsf{O}$-module structure we get a quadratic space over $\mathbb{Z}$. Consider the sub-$\mathbb{Z}$-module $I^{-1}\oplus I^{-1}$ (notice that this is not a right $\mathsf{O}$-module). Compute the set of short vectors $S=\{v \in I^{-1}\oplus I^{-1}| v^\dagger f v \leqslant p [I:\mathsf{O}]\}$.\\
        output $\{(f_1=\frac{f v v^\dagger f}{v^\dagger f v}, f_2=f-f_1)| v\in S\}  $
	}       
       \end{algorithm}
       \begin{proof}
       We prove the correctness of the algorithm. First we claim that the length of every vector in $I^{-1}\oplus I^{-1}$ is divisible by $p[I:\mathsf{O}]$. Indeed it suffices to show this for every localization at a prime number $l$. There exists an $a\in \mathsf{O}_{(l)}$ such that $I^{-1}=\mathsf{O}_{(l)} a$. Therefore every element in $I^{-1}_{(l)}\oplus I^{-1}_{(l)}$ is of the form $v a$ for $v\in \mathsf{O}_{(l)}^2$. This implies that $\nu_l ((va)^\dagger f (va))=\nu_l(N(a) v^\dagger f v)\geqslant \nu_l(N(a))+\nu_l(p)=\nu_l([\mathsf{O}:I^{-1}])+\nu_l(p)=\nu_l([I:\mathsf{O}])+\nu_l(p)$. Since $l$ was an arbitrary prime we conclude that the length of every vector in $I^{-1}\oplus I^{-1}$ is divisible by $p[I:\mathsf{O}]$. This implies that $S=\{v \in I^{-1}\oplus I^{-1}| v^\dagger f v = p [I:\mathsf{O}]\}$. In particular $S$ will either be empty or the set of shortest vectors.
       \par Let $v\in S$. Then we define $I_1=v I \subset \mathsf{O}^2$ and $I_2$ to be the $f$-orthogonal complement of $I_1$. Then $I_1,I_2$ are right $\mathsf{O}$-submodules of $\mathsf{O}^2$ satisfying i), ii), iii'). By Proposition \ref{decompLinAlg} this corresponds to a decomposition $f=f_1+f_2$ where $f_1$ is given by the formula $\frac{f v v^\dagger f}{v^\dagger f v}$.
       \par Conversely suppose we are given a decomposition $f=f_1+f_2$ with $I_1=\ker(f_1)\cong I$. Then $I_1$ will satisfy property iii') of Proposition \ref{decompLinAlg}. This implies that a generator of the free rank 1 $\mathsf{O}_L(I)$-module\footnote{See \cite[10.2.6]{Voight} for the definition of $\mathsf{O}_L(I)$, the left order of $I$.} $I_1 I^{-1}$ occurs in $S$. This proves that $\{(f_1=\frac{f v v^\dagger f}{v^\dagger f v}, f_2=f-f_1)| v\in S\}$ gives all pairs  as in Proposition \ref{decomp} such that $\ker(f_1)\cong I$. Notice that this list will contain some duplicate elements.
       \end{proof}
       \subsection{Reducible Fibers II: Role of 2-torsion}
		Using the same situation as in section \ref{SectionResultsMB} we want to further understand the reducible fibers of $\mathfrak{D}\rightarrow \mathbb{P}^1$. Notice that this is not fully answered by Proposition \ref{decomp}: If $D=E_1 \cup E_2$ is a reducible fiber of $\mathfrak{D}\rightarrow \mathbb{P}^1$, then $E_1$ and $E_2$ do not necessarily intersect in the origin but rather some point in $E^2[2]$. In this section we explain how to compute this $2$-torsion point. First we need a definition:
		\begin{definition}
		Let $A$ be an Abelian variety and $D$ a symmetric divisor on $A$. Define the function
		$e^D_*: A[2]\rightarrow \{\pm 1\}, \, P \mapsto (-1)^{m_D(P)-m_D(0)}$
		where $m_D(P)$ denotes the multiplicity of $D$ at $P$.
		\end{definition}
		The function $e^D_*$ has the following properties:
		\begin{prop} \label{e*}
		Let $A$ be an Abelian variety and $D$ a symmetric divisor on $A$. Denote $\mathcal{L}=\mathcal{O}(D)$.
		 \begin{enumerate}
		 
		\item[i)] $e^D_*$ depends only on $\mathcal{L}$.
		\item[ii)] $D$ is rationally equivalent to $0$ if and only if $D$ is algebraically equivalent to $0$ and $e^D_*(P)=1,\,  \forall P\in A[2]$.
		\item[iii)] $e^D_*$ is a quadratic function whose associated bilinear form is the restriction of the commutator pairing on $K(\mathcal{L}^2)$ to $A[2]$.
		\item[iv)] $e^D_*$ is locally constant in families, i.e., if $S$ is a $k$-scheme and $\mathfrak{D}\subset A\times_k S$ is a symmetric divisor flat over $S$, then for any $P\in A[2](k)$ the function $|S|\rightarrow \{\pm 1\},\, x\mapsto e^{\mathfrak{D}_x}_*(P)$ is locally constant.
		
		 \end{enumerate}
		\begin{proof}
		The first assertion is true because Mumford gives another definition for $e^D_*$ in terms of $\mathcal{L}$ and shows in \cite[§2, Proposition 2]{EqDefAV} that it is equivalent to our definition. The second assertion follows from \cite[§2, First Properties, iv)]{EqDefAV}. For the third assertion see \cite[Corollary 1, p. 314]{EqDefAV}. For the fourth assertion notice that it follows easily from Mumford's alternative definition that $|S|\rightarrow \{\pm 1\},\, x\mapsto e^{\mathfrak{D}_x}_*(P)$ is continuous in the Zariski topology. Here $\{\pm 1\}$ is considered as a closed subset of $\mathbb{G}_m$.
		\end{proof}
		\end{prop}
		\begin{cor}\label{RatEq} Two fibers of $\mathfrak{D}\rightarrow \mathbb{P}^1$ are rationally equivalent divisors in $E^2$.
	    \end{cor}
		\begin{proof}
		Indeed by Theorem \ref{MBPol} they both induce the polarization $\eta$. Therefore they are algebraically equivalent. Notice that by assumption $\mathfrak{D}$ is symmetric and thus Proposition \ref{e*} applies. By parts ii) and iv) we conclude that the fibers are rationally equivalent.
		\end{proof}		 
       	\begin{rem} Although one could infer the corollary from the fact that the base of $\mathfrak{D}\rightarrow \mathbb{P}^1$ is a rational curve, we prefer this more involved argument because it generalizes to Li's and Oort's higher dimensional theory.
       	\end{rem}
       	\begin{definition} 
       	
       	\begin{itemize}
       	 \item For $y\in \mathbb{P}^1(k)$ arbitrary we denote $\mathcal{L}=\mathcal{O}(\mathfrak{D}_y)$. By corollary \ref{RatEq} this line bundle is independent of $y$. The letter $\mathcal{L}$ is reserved for this line bundle for the rest of this paper.
       	\item Denote by
       	$$e^\mathfrak{D}_*: E^2[2] \rightarrow \{\pm 1\}$$
       	the function $e^\mathfrak{D}_*=e^{\mathfrak{D}_y}_*$ for any $y\in |\mathbb{P}^1|$.
	   	\end{itemize}       	
       	\end{definition}
       	We are now ready to describe the formula for the $2$-torsion point that is the intersection point of the components of a reducible fiber. Using the previous proposition this boils down to linear algebra over $\mathbb{F}_2$.
       	\par Consider a divisor $D=E_1\cup E_2$ as in Proposition \ref{decomp}. By Corollary \ref{RedQuot} there is a point $x\in \mathbb{P}^1(\mathbb{F}_{p^2})$ such that $\mathcal{A}_x\cong E_1^{(p)} \times  E_2^{(p)}$ with the product polarization.
       	\begin{lem}\label{2-tors}  Using the notation preceding the lemma there is a unique point $P\in E^2[2]$ such that $t_P(D)=\mathfrak{D}_x$. Furthermore $P$ can be computed as follows: 
       	Choose a basis $\mathfrak{B}$ of $E[2]^2$ as an $\mathbb{F}_2$-vector space. Take upper triangular matrices $\mathfrak{Q}, Q \in M_{4,4}(\mathbb{F}_2)$ such that the quadratic functions $e^{\mathfrak{D}}_*$ and $e^{D}_*$ are given by
		$e^\mathfrak{D}_*( (v)_\mathfrak{B})=(-1)^{v^t \mathfrak{Q} v},\, e^{D}_*( (v)_\mathfrak{B})=(-1)^{v^t Q v}$. $\mathfrak{Q}+Q$ is a diagonal matrix and denote its diagonal by the vector $d\in \mathbb{F}_2^4$. Then 
		$$(Q+Q^t)^{-1} d$$
		is the coordinate vector of the point $P$ in the basis $\mathfrak{B}$.
     	\end{lem}
     	\begin{proof}
		We know that $\mathcal{A}_x\cong E_1^{(p)} \times  E_2^{(p)}$ and the polarization $\lambda_x$ is the product polarization by corollary \ref{RedQuot}. Therefore $\lambda_x$ is induced by the symmetric divisor $E_1^{(p)} \times  \{0\}+\{0\} \times  E_2^{(p)}$. By assumption $\mathfrak{C}_x$ is another symmetric divisor inducing $\lambda_x$. Because two symmetric divisors inducing the same polarization are translates by a $2$-torsion point, there exists a $P\in A_x[2]$ such that $\mathfrak{C}_x=t_P(E_1^{(p)} \times  \{0\}+\{0\} \times  E_2^{(p)})$. Now $E^2[2]\cong A_x[2]$ canonically, so by abuse of notation denote the preimage of $P$ under this isomorphism also by $P$. Then $t_P(D)=\mathfrak{D}_x$.
\par This proves the existence of $P$. To prove uniqueness it suffices to establish the formula. Indeed it is clear that $e^\mathfrak{D}_*(\cdot)=e^{t_P(D)}(\cdot)=e^{D}_*(P+\cdot)$. We claim that $P$ is uniquely characterized by this property. This is a question about quadratic forms over $\mathbb{F}_2$. Define now $L=\mathcal{O}(E_1+E_2)$. Then $\mathcal{L}, L$ are ample symmetric line bundles inducing the polarization $\eta$. Thus $\mathcal{L}, L$ are algebraically equivalent and by part ii) of Proposition \ref{e*} one has $\mathcal{L}^2\cong L$. Part iii) of Proposition \ref{e*} gives:
$$e_*^\mathfrak{D}(P_1+P_2)=e_*^\mathfrak{D}(P_1) e_*^\mathfrak{D}(P_2) e^{\mathcal{L}^2}(P_1,P_2)\, \forall P_1,P_2\in E^2[2]$$
and similarly for $D$. Since $\mathcal{L}^2\cong L^2$, this implies that the bilinear forms associated to the quadratic functions $e_*^\mathfrak{D}, e_*^{D}$ agree. From that one deduces $\mathfrak{Q}+\mathfrak{Q}^t=Q+Q^t$ and thus $\mathfrak{Q}+Q$ is indeed a diagonal matrix. 
\par On the other hand $e^{\mathcal{L}^2}(\cdot, \cdot)_{|{E^2[2]}}$ is non-degenerate because it is the restriction of a non-degenerate pairing to the $2$-primary part. Therefore $Q+Q^t$ is invertible. We conclude that $P=((Q+Q^t)^{-1} d)_\mathfrak{B}$ is well-defined. An elementary computation shows that $P$ is the unique point in $E^2[2]$ satisfying 
$e^\mathfrak{D}_*(\cdot)=e^{D}_*(P+\cdot)$
     	\end{proof}
       	\subsection{Reducible Fibers III: Theta groups}
		In this section we continue our study of the reducible fibers of the map $\mathfrak{D}\rightarrow \mathbb{P}^1$. For the main algorithm it will be important to understand how two reducible fibers can be used to pin down an isomorphism of $\mathcal{G}(\mathcal{L})$ with the Heisenberg group $\mathcal{G}(\alpha_p)$.
		\par Suppose we are given two distinct reducible fibers $\mathfrak{D}_x, \mathfrak{D}_{x'}$. Then Corollary \ref{RedQuot} gives a subgroup scheme $\mathcal{H} \subset E^2$ isomorphic to $\alpha_p$ such that $\mathfrak{D}_x$ is $\mathcal{H}$-invariant. Similarly we get $\mathcal{H}' \subset E^2$ isomorphic to $\alpha_p$ such that $\mathfrak{D}_{x'}$ is $\mathcal{H}'$-invariant. By Proposition \ref{LevelPullDiv} $\mathcal{H}$ and $\mathcal{H}'$ uniquely lift to level subgroups $\tilde{\mathcal{H}}$ (resp. $\tilde{\mathcal{H}'}$) inside $\mathcal{G}(\mathcal{L})$.
      \begin{prop}\label{RedTheta} Use the notation preceding the proposition. For any isomorphism $\mathcal{H}\stackrel{\sim}{\rightarrow} \alpha_p$ there is a unique isomorphism $\phi:\mathcal{G}(\mathcal{L})\stackrel{\sim}{\rightarrow} \mathcal{G}(\alpha_p)$ such that $\phi(\tilde{\mathcal{H}})=\{1\}\times \alpha_p \times \{0\}, \phi(\tilde{\mathcal{H}'})=\{1\}\times \{0\} \times \alpha_p$ and the diagram
      $$\begin{tikzcd}
		\tilde{\mathcal{H}} \arrow{r}{\phi_{|\tilde{\mathcal{H}}}}\arrow{d} & \{1\}\times \alpha_p \times \{0\} \arrow[equal]{d}\\
		\mathcal{H} \arrow{r}{\sim} & \alpha_p
      \end{tikzcd}$$
       commutes.
      \end{prop}
		\begin{proof}
		Immediate consequence of Lemma \ref{Heis} and Corollary \ref{RedQuot}.
\end{proof}		      
	\begin{rem}\label{InduIso}
      Notice that we do not fix an isomorphism $\mathcal{H}' \cong \alpha_p$, instead such an isomorphism is then determined by $\phi$. In our thesis we will explain how this isomorphism can be computed. One has to make the commutator pairing on $\mathcal{G}(\mathcal{L})$ explicit. There are two distinct formulas for this computation. The first involves certain translation invariant vector fields on $E^2$. The second uses the action of Frobenius and certain endomorphisms of $E$ on de Rham cohomology.
      \end{rem}
	
\subsection{Construction of irreducible fibers}
We drop now the assumption that $k$ is algebraically closed. Instead $k$ is any field containing $\mathbb{F}_{p^2}$. In this section we will explain our algorithm that constructs all the irreducible fibers of a Moret-Bailly family. To define a Moret-Bailly family one needs an hermitian positive definite matrix $f\in M_{2,2}(\mathsf{O})$ inducing a polarization $\eta$ on $E^2$ with kernel $E^2[F]$. The basic idea is to use the information obtained from two reducible fibers of the family $\mathfrak{D} \rightarrow \mathbb{P}^1$. In order to obtain two reducible fibers we need to decompositions $f=f_1+f_2=f_1'+f_2'$ as in Proposition \ref{decomp}. The fibers will be distinct if and only if the sets $\{f_1,f_2\}$ and $\{f_1',f_2'\}$ are distinct by Corollary \ref{decompMB}.
\par Now given $\{f_1,f_2\}$ (resp. $\{f_1',f_2'\}$) we can compute the corresponding reducible fiber $D_0$ (resp. $D_\infty$). From the divisor $D_0$ we can construct a subgroup scheme $\mathcal{H}\cong \alpha_p$ of $E^2$ using Corollary \ref{RedQuot}. Then using Lemma \ref{alpha} we can choose an isomorphism of Hopf algebras $A_{\mathcal{H}}\cong k[z]/(z^p)$ such that the comultiplication is given by $z\mapsto z\otimes 1 + 1 \otimes z$. The next theorem explains how to compute the other fibers $D_x$ for $x=[t:1]\in \mathbb{P}^1(k)\setminus \{ [1:0]\}$. Indeed since all the $D_x$ are rationally equivalent it suffices to determine a rational function $g_x\in K(E^2)$ such that $(g_x)=D_x-D_0$. To start with there exists a rational function $g\in K(E^2)$ such that $(g)=D_{\infty}- D_0$. We claim that we can compute $g_x$ from $g$:        
\begin{thm}\label{Main} In the notation preceding the Theorem let $g_\mathcal{H} \in K(E^2)\otimes_k A_{\mathcal{H}}$ denote the image of $g$ under the map $K(E^2)\rightarrow K(E^2)\otimes_k A_{\mathcal{H}}$ describing the translation action of $\mathcal{H}$ on $E^2$.\\
        Then if one takes $g_x\in K(E^2)$ to be the image of $g_\mathcal{H}$ under the $K(E^2)$-linear map
        $$K(E^2)\otimes A_{\mathcal{H}} \rightarrow K(E^2)$$
        $$\rho \otimes z^{2i}\mapsto \frac{(2i)! (-t)^i \rho}{i! 2^i},\, \rho \otimes z^{2i+1} \mapsto 0$$
        one has $(g_x)= D_x-D_0$
\end{thm}
\begin{proof}
From the divisor $D_\infty$ we can construct a subgroup scheme $\mathcal{H}'\subset E^2$ using Corollary \ref{RedQuot}. Consider the line bundle $\mathcal{L}=\mathcal{O}(D_0)$. There exists a unique level subgroup over $\mathcal{H}$ (resp. $\mathcal{H}'$) in $\mathcal{G}(\mathcal{L})$ that we will by abuse of notation denote using the same letter. From Proposition \ref{RedTheta} we get an isomorphism $\mathcal{G}(\mathcal{L})\cong \mathcal{G}(\mathcal{H})=\mathbb{G}_m\times \mathcal{H} \times \mathcal{H}^D$ that maps $\mathcal{H}$ to $\{1\}\times \mathcal{H} \times \{0\}$ and $\mathcal{H}'$ maps to $\{1\} \times \{0\} \times \mathcal{H}^D$. By Theorem \ref{ThetaRep} there is an absolutely irreducible weight 1 representation of $\mathcal{G}(\mathcal{L})$ on $\HH^0(E^2,\mathcal{L})$. By Theorem \ref{ThetaRepUniq} (or Riemann-Roch) we have $\dim_k(\HH^0(E^2,\mathcal{L}))=p$. On the other hand $\mathcal{G}(\mathcal{H})$ also has a weight 1 representation on the $p$-dimensional $k$-vector space $A_{\mathcal{H}^D}$ given by
$$((c, h, \chi)f)(\chi_0)=c \cdot \chi_0(h)\cdot  f(\chi+\chi_0)$$
for any $k$-scheme $T$ and $c\in \mathbb{G}_m(T), h\in \mathcal{H}(T), \chi\in \mathcal{H}^D(T), \chi_0\in \mathcal{H}^D(T), f\in A_{\mathcal{H}^D}\otimes \HH^0(T, \mathcal{O}_T)= \Hom_{\text{Sch}}(\mathcal{H}^D_T, \mathbb{A}^1_T)$. Then Theorem \ref{ThetaRepUniq} implies that there exists an isomorphism of $k$-vector spaces
$$\psi: A_{\mathcal{H}^D}\stackrel{\sim}{ \longrightarrow} \HH^0(E^2, \mathcal{L})$$
that identifies the two representations. By Schur's lemma $\psi$ is uniquely determined up to a scalar.
\par We have $(A_{\mathcal{H}^D})^{\mathcal{H}^D}=k$. On the other hand we claim that $\HH^0(E^2, \mathcal{L})^{\mathcal{H}'}=k g$. Indeed $g$ gives an element in $\HH^0(E^2, \mathcal{L})=\HH^0(E^2, \mathcal{O}(D_0))$. Furthermore $g$ is $\mathcal{H}'$-invariant by Proposition \ref{LevelPullDiv} because it maps to $1$ under the isomorphism $\HH^0(E^2, \mathcal{O}(D_0))\stackrel{\sim}{\rightarrow} \HH^0(E^2, \mathcal{O}(D_\infty)),\, \rho \mapsto g^{-1} \rho$. This proves the claim. Now the claim implies that we can without loss assume $\psi(1)=g$. Consider now the action of $\mathcal{H}$ on $A_{\mathcal{H}^D}$ given by restricting the representation of $\mathcal{G}(\mathcal{H})$. This determines a map
$$A_{\mathcal{H}^D}\rightarrow A_{\mathcal{H}^D}\otimes A_{\mathcal{H}}\,.$$
It follows from the definition of the representation of $\mathcal{G}(\mathcal{H})$ that this map is given by multiplying with the element in $A_{\mathcal{H}^D}\otimes A_{\mathcal{H}}=A_{\mathcal{H}}^\vee\otimes A_{\mathcal{H}}$ which comes from Cartier duality. To make this explicit notice that we have chosen an isomorphism $\mathcal{H}\cong \alpha_p$. The autoduality of $\alpha_p$ gives an isomorphism $\mathcal{H}^D\cong \alpha_p$. This gives a generator of $A_{\mathcal{H}^D}$ that we will also denote by $z$ (this can not cause any confusion). Then the map
$$A_{\mathcal{H}^D}\rightarrow A_{\mathcal{H}^D}\otimes A_{\mathcal{H}}$$
is given by multiplication with $\exp(z\otimes z)\in A_{\mathcal{H}^D}\otimes A_{\mathcal{H}}$. In particular $1$ maps to $\exp(z\otimes z)$. Here $\exp(\cdot)$ denotes the usual exponential series. Notice that the expression $\exp(z\otimes z)$ is well-defined although $\text{char}(k)=p$ since $z^p=0$.
\par On the other hand $\mathcal{H}$ also acts on $\HH^0(E^2, \mathcal{L})$ by restricting the action of $\mathcal{G}(\mathcal{L})$. This gives a map
$$\HH^0(E^2, \mathcal{L})\rightarrow \HH^0(E^2, \mathcal{L}) \otimes A_{\mathcal{H}}\,.$$
By Proposition \ref{ThetaRepPullDiv} this action is given by the translation action on $K(E^2)$ using the identification
$$\HH^0(E^2,\mathcal{L})=\HH^0(E^2, \mathcal{O}(D_0))=\{\rho \in K(E^2)| (\rho)+D_0 \text{ is effective}\}\,.$$
In particular $g$ maps to $g_{\mathcal{H}}$ under this map. Because the isomorphism $\psi$ is equivariant, one has $(\psi\otimes \id_{A_{\mathcal{H}}})(\exp(z\otimes z))=g_{\mathcal{H}}$.
\par We will now focus on the point $x=[t:1]\in \mathbb{P}^1$. Indeed this point determines a map $\alpha_p \rightarrow \alpha_p \times \alpha_p,\, s\mapsto (ts, s)$. Let us denote by $\mathcal{H}_x$ the image of the composition $\alpha_p\rightarrow \alpha_p \times \alpha_p \cong \mathcal{H}\times \mathcal{H}^D$. By Lemma \ref{MB} and Proposition \ref{LevelPull} there is a unique level subgroup in $\mathcal{G}(\mathcal{H})$ above $\mathcal{H}_x$. We will denote by abuse of notation this level subgroup by $\mathcal{H}_x$. Its image under the isomorphism $\mathcal{G}(\mathcal{H}) \cong \mathcal{G}(\mathcal{L})$ is also denoted $\mathcal{H}_x$. Now the desired rational function $g_x$ will be a generator of the 1-dimensional vector space $\HH^0(E^2, \mathcal{L})^{\mathcal{H}_x}$ (see Proposition \ref{ThetaRepPull}). On the other hand $(A_{\mathcal{H}^D})^{\mathcal{H}_x}$ is generated by the element $\exp(\frac{-tz^2}{2})$, see \cite[proof of Proposition 1.5]{Moret-Bailly}. Therefore we can take $g_x=\psi(\exp(\frac{-t z^2}{2}))$. The theorem follows because $\exp(\frac{-t z^2}{2})\in A_{\mathcal{H}^D}$ is the image of $\exp(z\otimes z)$ under the map
$A_{\mathcal{H}^D}\otimes A_{\mathcal{H}} \rightarrow A_{\mathcal{H}^D}$\\
$\rho\otimes z^{2i}\mapsto  \frac{(2i)! (-t)^i \rho}{i! 2^i},\, \rho\otimes z^{2i+1} \mapsto 0$.
\end{proof}
\begin{rem} Notice that the coordinate system on $\mathbb{P}^1$ used in the previous theorem is not the same as the coordinate system in section 4.1. For the coordinate system in section 4.1 a point $[t_0:t_1]$ corresponds to the embedding $\alpha_p \stackrel{(t_0,t_1)}{\hookrightarrow} \alpha_p^2 \cong E^2[F]$, where the last isomorphism comes from an isomorphism $\alpha_p \cong E[F]$.
\par On the other hand for the coordinate system in the previous theorem a point $[t_0:t_1]$ corresponds to the embedding $\alpha_p \stackrel{(t_0,t_1)}{\hookrightarrow} \alpha_p^2 \cong \mathcal{H}\times \mathcal{H}'=E^2[F]$, where the isomorphism $\alpha_p \cong \mathcal{H}'$ is the one induced by Proposition \ref{RedTheta} from the choice of an isomorphism $\alpha_p\cong \mathcal{H}$.
\par With the first coordinate system one can directly write down the image of $\alpha_p \hookrightarrow E^2$ as a closed subgroup scheme. This is not possible with the second coordinate system. It is possible to compute the change of coordinates matrix, see also Remark \ref{InduIso}.
\end{rem}
We are ready to explain the main algorithm.\\
\begin{algorithm}[H]\label{AlgMB}
\caption{Irreducible fibers of a Moret-Bailly family\label{gen2}}
\DontPrintSemicolon
        \SetKwInOut{Input}{input}
        \SetKwInOut{Output}{output}
        \Input{An hermitian positive definite matrix $f\in M_{2,2}(\mathsf{O})$ inducing a polarization $\eta$ on $E^2$ with kernel $E^2[F]$. Two decompositions $f=f_1+f_2=f_1'+f_2'$ as in Proposition \ref{decomp} corresponding to distinct divisors. A point $x\in \mathbb{P}^1(k)$ such that the fiber $\mathfrak{C}_x$ of the Moret-Bailly family $\mathfrak{C}\rightarrow \mathbb{P}^1$ constructed from $(E^2,\eta)$ is irreducible.}
        \Output{The curve $\mathfrak{C}_x$}
        \Begin{
		$E_i\longleftarrow \ker(\gamma(f_i)),\, E_i'\longleftarrow \ker(\gamma(f_i')),\, i=1,2$\\
		$\mathfrak{D}_0 \longleftarrow E_1+E_2,\, D_{\infty} \longleftarrow E_1'+E_2'$ as a divisor on $E^2$.\\
		Compute the $2$-torsion point $P$, such that $t_P(D_{\infty})\sim \mathfrak{D}_0$ using the formula of Lemma \ref{2-tors}.\\
		$\mathfrak{D}_\infty \longleftarrow t_P(D_\infty)$.\\
        Compute a rational function $g\in K(E^2)$ such that $(g)=\mathfrak{D}_\infty- \mathfrak{D}_0$.\\
        Compute the rational function $g_x$ using the formula from Theorem \ref{Main}.
        The vanishing locus
        $$\mathcal{V}(g_x)\subset E^2$$
        is the divisor $\mathfrak{D}_x$.\\
        Put $\mathfrak{C}_x$ equal to $\mathfrak{D}_x^{\text{norm},(p)}$, i.e., the Frobenius twist of the normalization of $\mathfrak{D}_x$.\\
        Output: $\mathfrak{C}_x$
        }
\end{algorithm}
\begin{proof}
We prove the correctness of the algorithm. Indeed Theorem \ref{Main} gives us a rational function $g_x$ such that $(g_x)=\mathfrak{D}_x - \mathfrak{D}_0$. Therefore $\mathfrak{D}_x = \mathcal{V}(g_x)$. The map $\mathfrak{D}_x\rightarrow \mathfrak{C}_x$ is purely inseparable of degree $p$. By \cite[Tag 0CCX]{Stacks} there is an isomorphism $\mathfrak{C}_x \cong  \mathfrak{D}_x^{\text{norm},(p)}$.
\end{proof}
\begin{rem} The algorithm constructs the divisor $\mathfrak{D}_x$ without making any field extensions. Indeed notice that the condition $F^2+p=0$ implies that all endomorphism and all the 2-torsion points of $E$ are defined over $\mathbb{F}_{p^2}$ because the $p^2$-Frobenius acts trivially. It follows that $\mathfrak{D}\subset E^2$ is defined over $\mathbb{F}_{p^2}$ as a closed subscheme. This implies that $\mathfrak{C}$ and the map $\mathfrak{C}\rightarrow \mathbb{P}^1$ are also defined over $\mathbb{F}_{p^2}$. The algorithm correctly computes the fibers of this map up to isomorphism over the ground field (not over its algebraic closure).
\end{rem}
\begin{rem}
Notice that steps 1-6 are independent of $x$. Thus after performing steps 1-6 we can quickly compute many fibers.
\par In fact we could theoretically even construct the generic fiber (or a model for $\mathfrak{C}\rightarrow \mathbb{P}^1$) by applying the algorithm to the generic point. However this is very expensive in practice because normalizing the very singular curve $\mathfrak{D}_x$ will take too long over a function field. By replacing step 8 with a technique that avoids normalization one can remedy this (see the author's PhD thesis for details).
\end{rem}
\subsection{Examples}
We have implemented Algorithms \ref{AlgDec} and \ref{AlgMB} in MAGMA. The source code is available on GitHub via\\
\url{https://github.com/Andreas-Pieper/Supersingular}. We present some examples that we computed using these programs.
\begin{exmp} 
Consider $p= 1601$ and $B=\left( \frac{-3,-p}{\mathbb{Q}} \right)=\langle 1,i,j,k \rangle_{\mathbb{Q}}$ with multiplication given by $i^2=-3,j^2=-p, ij=k=-ji$ and the maximal order $\mathsf{O}=\Large\langle 1, \frac{1}{2} + \frac{1}{2}i, \frac{1}{2} + \frac{1}{6} i + \frac{1}{2}j + \frac{1}{6}k, -\frac{1}{2} + \frac{1}{6}i - \frac{1}{3} k   \Large\rangle_{\mathbb{Z}} $ of $B$. This maximal order contains $j$, which generates a two-sided ideal of reduced norm $p$, and thus there exists a supersingular elliptic curve $E/\overline{\mathbb{F}}_p$ defined over $\mathbb{F}_p$ with endomorphism algebra $\mathsf{O}$ (see \cite[Lemma 42.4.1]{Voight}). The geometric Frobenius will satisfy $F^2+p=0$ because of the Hasse-Weil bound. Choose the matrix
$$f=\begin{pmatrix}
1601 & 40j\\
-40j & 1601
\end{pmatrix}\in M_{2,2}(\mathsf{O})\,.$$
We apply Algorithm \ref{AlgDec} to $f$ and all the right ideal classes of $\mathsf{O}$ and find $8000=5p-5$ decompositions $f=f_1+f_2$ as in Proposition \ref{decompLinAlg}. The first few are:
$$
f_1\in \Big\lbrace \begin{pmatrix}
1601 & 40j\\
-40j & 1600
\end{pmatrix}, \, \begin{pmatrix}
1054 & \frac{27}{2} - \frac{1}{6}i + \frac{53}{2}j - \frac{1}{6}k\\
 \frac{27}{2} + \frac{1}{6}i - \frac{53}{2} j + \frac{1}{6}k & 1067
\end{pmatrix}$$
$$ \begin{pmatrix}
1094 & 13 + 27j\\
13 - 27j & 1067
\end{pmatrix}, \, \begin{pmatrix}
1054 & \frac{27}{2} + \frac{1}{6} i + \frac{53}{2}j + \frac{1}{6} k\\
\frac{27}{2} - \frac{1}{6} i - \frac{53}{2} j - \frac{1}{6} k & 1067
\end{pmatrix}, \ldots\Big\rbrace\,.$$
The computation took 4min 50s on a 3.10GHz dual Intel processor.
\end{exmp}
\begin{exmp} Consider $p=5$. The field extension $\mathbb{F}_{5^2}/\mathbb{F}_5$ is generated by a primitive third root of unity $\zeta_3\in \mathbb{F}_{5^2}$. We will take the supersingular elliptic curve $E/ \overline{\mathbb{F}}_5: y^2=x^3+1$. Its endomorphism algebra is a maximal order in
$$B=\mathbb{Q}[\zeta_3, F]/(\zeta_3^2+\zeta_3+1, F^2+5, \zeta_3 F - F \zeta_3^2)$$
where $\zeta_3$ is the endomorphism $E\rightarrow E,\, (x,y)\mapsto (\zeta_3 x, y)$ and $F$ is the geometric Frobenius.
\par We apply Algorithm \ref{AlgMB} with input $$f=\begin{pmatrix}
5 & 2 F\\
-2F & 5
\end{pmatrix}\in M_{2,2}(\mathsf{O})$$
$$f_1=\begin{pmatrix}
4 & 2 F\\
-2F & 5
\end{pmatrix},\, f_2=\begin{pmatrix}
1& 0\\
0 & 0
\end{pmatrix}\in M_{2,2}(\mathsf{O})$$
$$f_1'=\begin{pmatrix}
5 & 2 F\\
-2F & 4
\end{pmatrix},\, f_2'=\begin{pmatrix}
0 & 0\\
0 & 1
\end{pmatrix}\in M_{2,2}(\mathsf{O})\,.$$
And $x=[a:1]\in \mathbb{P}^1(\mathbb{F}_{5^4})$, where $a\in \mathbb{F}_{5^4}$ is a root of the irreducible polynomial $x^4 + 4x^2 + 4x + 2 \in \mathbb{F}_5[x]$. 
The resulting curve $\mathfrak{C}_x$ is computed on a 3.10GHz dual Intel processor in 4.5s. We get the hyperelliptic model
$$\mathfrak{C}_x: y^2=(4a^3 + 3a^2 + 1) x^5 + (4 a^3 + a^2 + 4a + 3) x^4 + (3a^3 + 3a^2 + 2a) x^3+ (4a^3 + 4a^2)x^2+$$
$$(4a^3 + 2a^2 + 3a + 3)x + (3a^3 + 3a + 3)\,.$$
N.B.: The coefficients on the right are expressed in the $\mathbb{F}_5$-basis $a^3,a^2,a,1$ of $\mathbb{F}_{5^4}$. Replacing $a$ by another element of $\mathbb{F}_{5^4}$ does not give a curve with a supersingular Jacobian in general.
\end{exmp}
\begin{rem}
The matrices $f$ describing the polarization $\eta$ were found by hand in these examples. In general for a fixed $p$ there are finitely many choices for $f$ up to conjugation by $\text{Gl}_2(\mathsf{O})$. These are in bijection with the set of isomorphism classes of quaternion hermitian lattices in a fixed genus. There is an algorithm which computes a list of representatives for the isomorphism classes. This uses the so-called method of neighboring lattices. The method was developed by M. Kneser for the classification of quadratic lattices over $\mathbb{Z}$ (see \cite{Kneser1}).
\par The method generalizes to quaternion hermitian lattices by using the general strong approximation theorem of M. Kneser \cite{Kneser2} applied to the group of isometries of a quaternion hermitian lattice. The quaternion hermitian case is less technical because the latter group is simply-connected whereas $\text{SO}_n$ is not.
\end{rem}
\bibliographystyle{plain}
\bibliography{bibl.bib}
\end{document}